\documentclass{amsart}

\usepackage{amsmath}%
\usepackage{amsfonts}%
\usepackage{amssymb}%
\usepackage{graphicx}

\usepackage{graphicx}
\usepackage{color}

\newtheorem{theorem}{Theorem}
\theoremstyle{plain}

\newtheorem{corollary}[theorem]{Corollary}

\newtheorem{definition}[theorem]{Definition}

\newtheorem{lemma}[theorem]{Lemma}

\newtheorem{proposition}[theorem]{Proposition}
\newtheorem{remark}[theorem]{Remark}
\newtheorem{remarks}[theorem]{Remarks}

\numberwithin{equation}{section} \numberwithin{theorem}{section}
\begin{document}
\title[Critical problems]{The Brezis-Nirenberg problem for the fractional Laplacian with mixed Dirichlet-Neumann boundary conditions}

\author[E. Colorado]{Eduardo Colorado}
\address{Eduardo Colorado\hfill\break\indent
Departamento de Matem\'aticas\hfill\break\indent Universidad Carlos
III de Madrid\hfill\break\indent Avda. Universidad 30, 28911
Legan\'es (Madrid)\hfill\break\indent Spain
\hfill\break\indent \&
\hfill\break\indent Instituto de Ciencias Matem\'aticas, ICMAT (CSIC-UAM-UC3M-UCM)
\hfill\break\indent C/Nicol\'as Cabrera 15, 28049 Madrid\hfill\break\indent Spain
} \email{eduardo.colorado@uc3m.es \& eduardo.colorado@icmat.es}

\author[A. Ortega]{Alejandro Ortega}
\address{Alejandro Ortega\hfill\break\indent
Departamento de Matem\'aticas\hfill\break\indent Universidad Carlos
III de Madrid\hfill\break\indent Avda. Universidad 30, 28911
Legan\'es (Madrid)\hfill\break\indent Spain}
\email{alortega@math.uc3m.es}

\date{\today}

\maketitle

{\begin{center}\it {\small  Dedicated to Ireneo Peral on the occasion of his retirement}\end{center}}

\

\begin{abstract}
In this work we study the existence of solutions to the critical
Brezis-Nirenberg problem when one deals with the spectral fractional
Laplace operator and mixed Dirichlet-Neumann boundary conditions,
i.e.,
$$
\left\{\begin{array}{rcl}
(-\Delta)^su & = & \lambda u+u^{2_s^*-1},\quad u>0\quad\mbox{in}\quad \Omega,\\
u & = & 0\quad\mbox{on}\quad \Sigma_{\mathcal{D}},\\
\displaystyle\frac{\partial u}{\partial \nu} & = & 0\quad\mbox{on}\quad \Sigma_{\mathcal{N}},
\end{array}\right.
$$
where  $\Omega\subset\mathbb{R}^N$ is a regular bounded domain,
$\frac{1}{2}<s<1$, $2_s^*$ is the critical fractional Sobolev exponent,
$0\le\lambda\in \mathbb{R}$, $\nu$ is the outwards normal to $\partial\Omega$,  $\Sigma_{\mathcal{D}}$,
$\Sigma_{\mathcal{N}}$ are smooth $(N-1)$-dimensional submanifolds
of $\partial\Omega$ such that
$\Sigma_{\mathcal{D}}\cup\Sigma_{\mathcal{N}}=\partial\Omega$,
$\Sigma_{\mathcal{D}}\cap\Sigma_{\mathcal{N}}=\emptyset$, and
$\Sigma_{\mathcal{D}}\cap\overline{\Sigma}_{\mathcal{N}}=\Gamma$ is
a smooth $(N-2)$-dimensional submanifold of $\partial\Omega$.
\end{abstract}

\section{Introduction}\label{sec:1}

For the last decades Dirichlet and Neumann boundary problems
associated with elliptic equations as
\begin{equation}\label{eliptica}
-\Delta u=f(x,u)
\end{equation}
have been widely investigated with different nonlinearities
$f(x,u)$. In contrast, mixed Dirichlet-Neumann boundary problems
have been much less investigated. Nevertheless, some important
results dealing with mixed Dirichlet-Neumann boundary problems
associated with \eqref{eliptica} have been proved over the years.
See \cite{ACP0,ACP,ColP,CP2,D,De1,De2,G,M,S} among others.

Problems associated with \eqref{eliptica}, substituting the operator
by the fractional Laplacian, have been extensively investigated in
the last years, with Dirichlet or Neumann boundary conditions (cf.,
e.g., \cite{BCdPS,BCSS,BrCdPS,CS,CT,DnPV,DRoV,MbS,SV,T}, among
others), but these fractional elliptic  problems, once again,
 have not been
so much investigated with mixed Dirichlet-Neumann boundary data, cf.
\cite{BM,CERW,DRoV}. Indeed, up to our knowledge, there are no
references for mixed Dirichlet-Neumann boundary problems involving
the spectral fractional Laplacian operator, which is the one we deal
with here. Precisely, we study the Brezis-Nirenberg problem, cf.
\cite{BN}, with the spectral fractional Laplacian operator
associated with mixed Dirichlet-Neumann boundary data. A turning
point in the history  of elliptic boundary problems associated with
\eqref{eliptica} was the seminal paper by Brezis and Nirenberg
\cite{BN}, where the critical power problem for the classical
Laplacian with a lower-order perturbation term and a Dirichlet
boundary condition was studied. For the pure critical problem it is
well known that there is no positive solution when the domain is
star-shaped due to a Pohozaev identity, cf. \cite{Po}. Nevertheless,
Brezis and Nirenberg proved, among other results, that there exists
a  positive solution when the perturbation is linear, analyzing more
carefully the case when the domain is a ball. Since then, there have
arisen more than one thousand papers citing \cite{BN}. In the
fractional setting, Brezis-Nirenberg problems have been also widely
investigated. For brevity we just cite some related works dealing
only with the fractional Laplacian, cf., e.g., \cite{BCdPS,T} for
the spectral fractional Laplacian defined in \eqref{spectral}, and
\cite{MbS,SV} for the fractional Laplacian defined by a singular
integral in \eqref{poisson}; both with Dirichlet boundary condition.
As we said above, there are no references dealing with problems
involving the spectral fractional Laplacian and mixed
Dirichlet-Neumann boundary conditions. As a consequence, the main
goal of this manuscript is twofold: one is to address for the very
first time problems involving spectral fractional Laplacian together
with mixed Dirichlet-Neumann boundary conditions, and second to
prove existence of a positive solution for the Brezis-Nirenberg
problem in this fractional setting with mixed boundary conditions.

\

The precise problem we study in this work is the following,
\begin{equation}\label{problema}
        \left\{
        \begin{tabular}{lcl}
        $(-\Delta)^su=\lambda u+u^{2_s^*-1}$ & &in $\Omega\subset \mathbb{R}^{N}$, \\
        $u>0$& & in $\Omega$,\\
        $B(u)=0$  & &on $\partial\Omega=\Sigma_{\mathcal{D}}\cup\Sigma_{\mathcal{N}}$, \\
        \end{tabular}
        \right.
        \tag{$P_\lambda$}
\end{equation}
where $\frac{1}{2}<s<1$, $\Omega$ is a smooth bounded domain of $\mathbb{R}^N$, $N>2s$, and mixed Dirichlet-Neumann boundary conditions
of the form
\begin{equation}\label{mixed}
B(u)=u\chi_{\Sigma_{\mathcal{D}}}+\frac{\partial u}{\partial
\nu}\chi_{\Sigma_{\mathcal{N}}},
\end{equation}
where $\nu$ is the outwards normal to $\partial\Omega$, $\chi_A$ stands for the characteristic function of a set $A$, $\Sigma_{\mathcal{D}}$ and
$\Sigma_{\mathcal{N}}$ are smooth $(N-1)$-dimensional submanifolds of $\partial\Omega$ such that
$\Sigma_{\mathcal{D}}$ is a closed submanifold of $\partial\Omega$, $\mathcal{H}_{N-1}(\Sigma_{\mathcal{D}})=\alpha>0$ where $\mathcal{H}_{N-1}$ is the
$(N-1)$-dimensional Hausdorff measure, $\Sigma_{\mathcal{D}}\cap\Sigma_{\mathcal{N}}=\emptyset$,
$\Sigma_{\mathcal{D}}\cup\Sigma_{\mathcal{N}}=\partial\Omega$ and $\Sigma_{\mathcal{D}}\cap\overline{\Sigma}_{\mathcal{N}}=\Gamma$ is a smooth
$(N-2)$-dimensional submanifold.

\

For the Dirichlet case ($\mathcal{H}_{N-1}(\Sigma_{\mathcal{N}})=0$) it can be seen (\cite{BrCdPS}) that using a generalized Pohozaev identity,
 problem $(P_{\lambda})$ has no solution for $\lambda=0$ and $\Omega$ a star-shaped domain.
As we will see, in the mixed boundary data case the situation is different.

The classical Pohozaev's identity was extended to the mixed Dirichlet-Neumann boundary data case, involving the classical Laplace operator by
Lions-Pacella-Tricarico \cite{LPT}. Following that ideas, we extend that result to our mixed fractional setting.
Precisely, as in \cite{ACP,ColP}, we will show that taking the mixed Dirichlet-Neumann boundary conditions, in an appropriate way, problem
$(P_\lambda)$ has a solution when $\lambda=0$, in contrast to the Dirichlet case. Thus, we can include the value $\lambda =0$ in the existence results. The
main result proved in this paper is the following.
\begin{theorem}\label{th_existencia}
Assume that  $\frac 12<s<1$ and $N\geq 4s$. Let $\lambda_{1,s}$ be
the first eigenvalue of the fractional operator $(-\Delta)^s$ with
mixed Dirichlet-Neumann boundary conditions \eqref{mixed}. Then
problem $(P_{\lambda})$
\begin{enumerate}
\item has no solution for $\lambda\geq\lambda_{1,s}$,
\item has at least one solution for $0<\lambda<\lambda_{1,s}$,
\item has at least one solution for $\lambda=0$ and $\mathcal{H}_{N-1}(\Sigma_{\mathcal{D}})$ small enough.
\end{enumerate}
\end{theorem}
Note that the range $\frac 12<s<1$ is natural for mixed boundary problems in our fractional setting, see Remark \ref{rem:rango_s}.

\

\textbf{Organization of the paper}. This manuscript have four more
sections. In Section \ref{sec:2} we establish the appropriate
functional setting for the study of problem $(P_{\lambda})$,
including the definition of an auxiliary problem introduced by
Caffarelli and Silvestre, \cite{CS}, that will help us to overcome
some difficulties that appear when we deal with the fractional
operator. Following the ideas of \cite{G} and \cite{ACP}, we
introduce two constants $\widetilde{S}(\Sigma_{\mathcal{N}})$ and
$\widetilde{S}(\Sigma_{\mathcal{D}})$ respectively, that play a
similar role to that of the Sobolev constant in the celebrated paper
of Brezis and Nirenberg, \cite{BN}. In Section \ref{sec:3} we study
some useful properties of that constants. Section \ref{sec:4} is
devoted to prove Theorem \ref{th_existencia} and it is divided into
two subsections. In Subsection \ref{subsec:4.1} we prove the
statements $(1)$-$(2)$ in Theorem \ref{th_existencia}. In
 Subsection \ref{subsec:4.2}, we
use the constant $\widetilde{S}(\Sigma_{\mathcal{D}})$ to study the
existence of solution to problem $(P_\lambda)$ when we move the
boundary conditions in an appropriate way to be specified. These
results allow us to prove statement $(3)$ in Theorem
\ref{th_existencia}. Finally, in the last section we prove a
non-existence result by means of a Pohozaev-type identity.

\section{Functional setting and definitions}\label{sec:2}
The definition of the fractional powers of the positive Laplace
operator $(-\Delta)$, in a bounded domain $\Omega$ with homogeneous
mixed Dirichlet-Neumann boundary data, is carried out via the
spectral decomposition using the powers of the eigenvalues of
$(-\Delta)$ with the same boundary condition. Let
$(\varphi_i,\lambda_i)$ be the eigenfunctions (normalized with
respect to the $L^2(\Omega)$-norm) and eigenvalues of $(-\Delta)$
with homogeneous mixed Dirichlet-Neumann boundary data, then
$(\varphi_i,\lambda_i^s)$ are the eigenfunctions and eigenvalues of
$(-\Delta)^s$ with the same boundary conditions. Thus the fractional
operator $(-\Delta)^s$ is well defined in the space of functions
that vanish on $\Sigma_{\mathcal{D}}$,
\begin{equation*}
H_{\Sigma_{\mathcal{D}}}^s(\Omega)=\left\{u=\sum_{j\ge 1}
a_j\varphi_j\in L^2(\Omega):\
\|u\|_{H_{\Sigma_{\mathcal{D}}^s}(\Omega)}^2= \sum_{j\ge 1}
a_j^2\lambda_j^s<\infty\right\}.
\end{equation*}

As a direct consequence of the previous definition we get

\begin{equation}\label{spectral}
(-\Delta)^su=\sum_{j\ge 1} a_j\lambda_j^s\varphi_j,
\end{equation}
as well as
\begin{equation}\label{norma1}
\|u\|_{H_{\Sigma_{\mathcal{D}}^s}(\Omega)}=\|(-\Delta)^{\frac{s}{2}}u\|_{L^2(\Omega)}.
\end{equation}
This definition of the fractional powers of the Laplace operator allows us to integrate by parts in the appropriate spaces. A
 natural definition of energy solution to problem $(P_{\lambda})$ is the following.

\begin{definition}
We say that $u\in H_{\Sigma_{\mathcal{D}}}^s(\Omega)$ is a solution of $(P_{\lambda})$ if
\begin{equation}\label{energy_sol}
\int_{\Omega}(-\Delta)^{s/2}u(-\Delta)^{s/2}\psi dx=\int_{\Omega}\left(\lambda u+u^{2_s^*-1}\right)\psi dx,\ \ \text{for all}\ \psi\in H_{\Sigma_{\mathcal{D}}}^s(\Omega).
\end{equation}
\end{definition}
The right-hand side of \eqref{energy_sol} is well defined because of the embedding $H_{\Sigma_{\mathcal{D}}}^s(\Omega)\hookrightarrow L^{2_s^*}(\Omega)$ while $u\in H_{\Sigma_{\mathcal{D}}}^s(\Omega)$ so $\lambda u+u^{2_s^*-1}\in L^{\frac{2N}{N+2s}}\hookrightarrow \left(H_{\Sigma_{\mathcal{D}}}^s(\Omega)\right)'$.
The energy functional associated with problem $(P_{\lambda})$ is
\begin{equation}\label{energy_functional}
I(u)=\frac{1}{2}\int_{\Omega}|(-\Delta)^{s/2}u|^2dx-\frac{\lambda}{2}\int_{\Omega}u^2dx-\frac{N-2s}{2N}\int_{\Omega}u^{\frac{2N}{N-2s}}dx.
\end{equation}
This functional is well defined in
$H_{\Sigma_{\mathcal{D}}}^s(\Omega)$ and critical points of $I$,
defined by \eqref{energy_functional}, correspond to solutions of
$(P_{\lambda})$.

\begin{remark}\label{rem:rango_s}
As it was proved in \cite{BSV}, for the range $0<s\le \frac 12$,
$H_0^s(\Omega)=H^s(\Omega)$, and for $\frac 12<s<1$,
$H_0^s(\Omega)\subsetneq H^s(\Omega)$. As a consequence,
$H_{\Sigma_{\mathcal{D}}}^s(\Omega)=H^s(\Omega)$ for $0<s\le \frac
12$. This is the reason why we work here with the fraction $\frac
12<s<1$, in which $H_{\Sigma_{\mathcal{D}}}^s(\Omega)\subsetneq
H^s(\Omega)$.
\end{remark}

In order to overcome some difficulties that appear along several proofs in the paper we use the ideas of Caffarelli and Silvestre, \cite{CS}, together
with those of \cite{BrCdPS} to give an equivalent definition of the operator $(-\Delta)^s$ defined in a bounded domain by means of an auxiliary problem.
Associated with the domain $\Omega$, we consider the cylinder $\mathcal{C}_{\Omega}=\Omega\times(0,\infty)\subset\mathbb{R}_+^{N+1}$.
We denote with $(x,y)$ points that belongs to $\mathcal{C}_{\Omega}$ and with $\partial_L\Omega=\partial\Omega\times(0,\infty)$ the lateral boundary of
the extension cylinder. Given a function $u\in H_{\Sigma_{\mathcal{D}}}^s(\Omega)$, we define its $s$-extension $w=E_{s}[u]$ to the cylinder
$\mathcal{C}_{\Omega}$ as the solution of the problem

\begin{equation}\label{extension}
        \left\{
        \begin{tabular}{lcl}
        $-div(y^{1-2s}\nabla w)=0$ & &in $\mathcal{C}_{\Omega}$, \\
        $B^*(w)=0$  & &on $\partial_L\mathcal{C}_{\Omega}$, \\
        $w(x,0)=u(x)$ & & in $\Omega\times\{y=0\}$,
        \end{tabular}
        \right.
\end{equation}
where
\begin{equation*}
B^*(w)=w\chi_{\Sigma_{\mathcal{D}}^*}+\frac{\partial w}{\partial \nu}\chi_{\Sigma_{\mathcal{N}}^*},
\end{equation*}
with $\Sigma_{\mathcal{D}}^*=\Sigma_{\mathcal{D}}\times(0,\infty)$ and $\Sigma_{\mathcal{N}}^*=\Sigma_{\mathcal{N}}\times(0,\infty)$. The extension function belongs to the space
\begin{equation*}
\mathcal{X}_{\Sigma_{\mathcal{D}}}^s(\mathcal{C}_{\Omega})=\overline{\mathcal{C}_{0}^{\infty}\big((\Omega\cup\Sigma_{\mathcal{N}})\times[0,\infty)\big)}^{\|\cdot\|_{\mathcal{X}_{\Sigma_{\mathcal{D}}}^s(\mathcal{C}_{\Omega})}},
\end{equation*}
equipped with the norm,
\begin{equation*}
\|z\|_{\mathcal{X}_{\Sigma_{\mathcal{D}}}^s(\mathcal{C}_{\Omega})}^2=\kappa_s\int_{\mathcal{C}_{\Omega}}y^{1-2s}|\nabla
z(x,y)|^2dxdy.
\end{equation*}
With that constant  $\kappa_s$, whose value can be consulted in \cite{BrCdPS}, the extension operator
between $H_{\Sigma_{\mathcal{D}}}^s(\Omega)$ and $\mathcal{X}_{\Sigma_{\mathcal{D}}}^s(\mathcal{C}_{\Omega})$ is an isometry, i.e.,
\begin{equation}\label{norma2}
\|E_s[\varphi]\|_{\mathcal{X}_{\Sigma_{\mathcal{D}}}^s(\mathcal{C}_{\Omega})}=\|\varphi\|_{H_{\Sigma_{\mathcal{D}}}^s(\Omega)},\
\text{for all}\ \varphi\in H_{\Sigma_{\mathcal{D}}}^s(\Omega).
\end{equation}
The key point of the extension function is that it is related to the fractional Laplacian of the original function through the formula
\begin{equation*}
\frac{\partial w}{\partial \nu^s}:= -\kappa_s \lim_{y\to 0^+} y^{1-2s}\frac{\partial w}{\partial y}=(-\Delta)^su(x).
\end{equation*}
In the case $\Omega=\mathbb{R}^N$ this formulation provides explicit expressions for both the fractional Laplacian and the $s$-extension in terms of the Riesz and the Poisson kernels, respectively. Namely,

\begin{equation}\label{poisson}
\begin{split}
w(x,y)&=P_y^s\ast u(x)=c_{N,s}y^{2s}\int_{\mathbb{R}^N}\frac{u(z)}{(|x-z|^2+y^2)^{\frac{N+2s}{2}}}dz\\
(-\Delta)^{s}u(x)&=d_{N,s}P.V.\int_{\mathbb{R}^N}\frac{u(x)-u(y)}{|x-y|^{N+2s}}.
\end{split}
\end{equation}
We refer to \cite{BrCdPS} in order to look up the values of the constants $\kappa_s$, $c_{N,s}$ and $d_{N,s}$ as well as the existent relation between
them, namely, $2s\kappa_sc_{N,s}=d_{N,s}$.
By the arguments above, we can reformulate our problem $(P_{\lambda})$ in terms of the extension problem as follows
\begin{equation}\label{extension_problem}
        %P_{\lambda}\equiv
        \left\{
        \begin{tabular}{lcl}
        $-div(y^{1-2s}\nabla w)=0$ & &in $\mathcal{C}_{\Omega}$, \\
        $B^*(w)=0$  & &on $\partial_L\mathcal{C}_{\Omega}$, \\
        $\frac{\partial w}{\partial \nu^s}=\lambda w+w^{2_s^*-1}$ & & in $\Omega\times\{y=0\}$.
        \end{tabular}
        \right.
        \tag{$P_{\lambda}^*$}
\end{equation}
An energy solution of this problem is a function $w\in \mathcal{X}_{\Sigma_{\mathcal{D}}}^s(\mathcal{C}_{\Omega})$ such that
\begin{equation*}
\kappa_s\int_{\mathcal{C}_{\Omega}} y^{1-2s}\langle\nabla w,\nabla\varphi \rangle dxdy=\int_{\Omega} \left(\lambda w(x,0)+w^{2_s^*-1}(x,0)\right)\varphi(x,0)dx,
\end{equation*}
for all
$\varphi\in\mathcal{X}_{\Sigma_{\mathcal{D}}}^s(\mathcal{C}_{\Omega})$.
Given $w\in
\mathcal{X}_{\Sigma_{\mathcal{D}}}^s(\mathcal{C}_{\Omega})$ a
solution to problem $(P_{\lambda}^*)$ the function
$u(x)=Tr[w](x)=w(x,0)$ belongs to the space
$H_{\Sigma_{\mathcal{D}}}^s(\Omega)$ and it is an energy solution to
problem $(P_{\lambda})$ and vice versa, if $u\in
H_{\Sigma_{\mathcal{D}}}^s(\Omega)$ is a solution to $(P_{\lambda})$
then $w=E_s[u]\in
\mathcal{X}_{\Sigma_{\mathcal{D}}}^s(\mathcal{C}_{\Omega})$ is a
solution to $(P_{\lambda}^*)$ and, as a consequence, both
formulations are equivalent. Finally, the energy functional
associated with problem $(P_{\lambda}^*)$ is the following,
\begin{equation}\label{extensionfunctional}
J(w)=\frac{\kappa_s}{2}\int_{\mathcal{C}_{\Omega}}y^{1-2s}|\nabla w|^2dxdy-\frac{\lambda}{2}\int_{\Omega}w^2dx-\frac{N-2s}{2N}\int_{\Omega}w^{2_s^*}dx.
\end{equation}
Plainly, critical points of $J$ in
$\mathcal{X}_{\Sigma_{\mathcal{D}}}^s(\mathcal{C}_{\Omega})$
correspond to critical points of $I$ in
$H_{\Sigma_{\mathcal{D}}}^s(\Omega)$. Moreover, minima of $J$ also
correspond to minima of $I$. The proof of this fact is similar to
the one of the Dirichlet case, see \cite{BCdPS}.

\

Also, in the Dirichlet case, there is a trace inequality \cite[Theorem 4.4]{BrCdPS}, i.e.,
\begin{equation}\label{sobext}
\int_{\mathcal{C}_{\Omega}}y^{1-2s}|\nabla z(x,y)|^2dxdy\geq C\left(\int_{\Omega}|z(x,0)|^rdx\right)^{\frac{2}{r}},
\end{equation}
for $1\leq r\leq\frac{2N}{N-2s},\ N>2s, z\in \mathcal{X}_{0}^s(\mathcal{C}_{\Omega})$, that  turns out to be very useful and by the previous
comments this inequality is equivalent to the fractional Sobolev inequality,
\begin{equation}\label{sobolev}
\int_{\Omega}|(-\Delta)^{s/2}v|^2dx\geq C\left(\int_{\Omega}|v|^rdx\right)^{\frac{2}{r}},
\end{equation}
for $1\leq r\leq\frac{2N}{N-2s},\ N>2s, v\in H_{0}^s(\Omega)$.
\begin{remark} When $r=2_s^*$ the best constant in \eqref{sobext} will be denoted by $S(s,N)$. This constant is explicit and independent of the domain $\Omega$, and its exact value is given by the following expression,
\begin{equation*}
S(s,N)=\frac{2\pi^s\Gamma(1-s)\Gamma(\frac{N+2s}{2})(\Gamma(\frac{N}{2}))^{\frac{2s}{N}}}{\Gamma(s)\Gamma(\frac{N-2s}{2})(\Gamma(N))^s}.
\end{equation*}
Since it is not achieved in any bounded domain (see Remarks
\ref{rem:cte_Sobolev}-(1)) we have that
\begin{equation*}
\int_{\mathbb{R}_{+}^{N+1}}\!\!y^{1-2s}|\nabla z(x,y)|^2dxdy\geq S(s,N)\left(\int_{\mathbb{R}^N}|z(x,0)|^{\frac{2N}{N-2s}}dx\right)^{\frac{N-2s}{N}}\!,\  z\in \mathcal{X}^s(\mathbb{R}_{+}^{N+1}).
\end{equation*}
Indeed, in the whole space case the latter inequality is achieved
when $z= E_s[u]$ and
\begin{equation*}
u(x)=u_{\varepsilon}(x)=\frac{\varepsilon^{\frac{N-2s}{2}}}{(\varepsilon^2+|x|^2)^{\frac{N-2s}{2}}},
\end{equation*}
with arbitrary $\varepsilon>0$, cf., \cite{BrCdPS}. Finally, the
best constant in \eqref{sobolev} with $\Omega=\mathbb{R}^N$ is given
by $\kappa_s S(s,N)$.
\end{remark}
In the mixed boundary data case the situation is quite similar
thanks to the fact that we are considering a Dirichlet condition on
$\Sigma_{\mathcal{D}}$ with
$0<\mathcal{H}_{N-1}(\Sigma_{\mathcal{D}})<\mathcal{H}_{N-1}(\partial\Omega)$,
hence there exists a positive constant $C$ such that
\begin{equation*}
0<C:=\inf_{\substack{u\in H_{\Sigma_{\mathcal{D}}}^s(\Omega)\\
u\not\equiv
0}}\frac{\|u\|_{H_{\Sigma_{\mathcal{D}}}^s(\Omega)}}{\|u\|_{L^{2_s^*}(\Omega)}},
\end{equation*}
so in terms of the extension function,
\begin{equation}\label{poinc}
\left(\int_\Omega\varphi^{\frac{2N}{N-2s}}(x,0)dx\right)^{\frac{N-2s}{2N}}\leq C
\|\varphi(\cdot,0)\|_{H_{\Sigma_{\mathcal{D}}}^s(\Omega)}=C
\|E_s[\varphi(\cdot,0)]\|_{\mathcal{X}_{\Sigma_{\mathcal{D}}}^s(\mathcal{C}_{\Omega})}.
\end{equation}
As we will see below this constant $C$ plays an important role in the proof of Theorem \ref{th_existencia}.
With this Sobolev-type inequality in hands we can prove the following result.

\begin{lemma}\label{lem:traceineq}
 Assume that $\varphi \in \mathcal{X}_{\Sigma_{\mathcal{D}}}^s(\mathcal{C}_{\Omega})$, then there exists a constant $C>0$ such that,
\begin{equation}\label{eq:traceineq}
\left(\int_\Omega
\varphi^{\frac{2N}{N-2s}}(x,0)dx\right)^{1-\frac{2s}N}\leq C
\int_{\mathcal{C}_{\Omega}} y^{1-2s} |\nabla \varphi|^2 dxdy.
\end{equation}
\end{lemma}
\begin{proof}
Thanks to \eqref{poinc} in order to prove \eqref{eq:traceineq} it
only remains to show the inequality
$\|E_s[\varphi(\cdot,0)]\|_{\mathcal{X}_{\Sigma_{\mathcal{D}}}^s(\mathcal{C}_{\Omega})}\leq\|\varphi
\|_{\mathcal{X}_{\Sigma_{\mathcal{D}}}^s(\mathcal{C}_{\Omega})}$.
This inequality is satisfied since, arguing as in \cite{BrCdPS},
\begin{align*}
\|\varphi \|_{\mathcal{X}_{\Sigma_{\mathcal{D}}}^s(\mathcal{C}_{\Omega})}^2=& \int_{\mathcal{C}_{\Omega}}y^{1-2s} |\nabla \varphi|^2dxdy\\
=&\int_{\mathcal{C}_{\Omega}}y^{1-2s} |\nabla \big(E_s[\varphi(\cdot,0)]+\varphi-E_s[\varphi(\cdot,0)]\big)|^2dxdy\\
=&\|E_s[\varphi(\cdot,0)]\|^2_{\mathcal{X}_{\Sigma_{\mathcal{D}}}^s(\mathcal{C}_{\Omega})}+ \|\varphi-E(\varphi(\cdot,0))\|^2_{\mathcal{X}_{\Sigma_{\mathcal{D}}}^s(\mathcal{C}_{\Omega})}\\
&+2\int_{\mathcal{C}_{\Omega}}y^{1-2s} \langle\nabla E_s[\varphi(\cdot,0)], \nabla(\varphi-E_s[\varphi(\cdot,0)])\rangle dxdy\\
=&\|E_s[\varphi(\cdot,0)]\|^2_{\mathcal{X}_{\Sigma_{\mathcal{D}}}^s(\mathcal{C}_{\Omega})}+
\|\varphi-E_s[\varphi(\cdot,0)]\|^2_{\mathcal{X}_{\Sigma_{\mathcal{D}}}^s(\mathcal{C}_{\Omega})}\\
&+2\int_{\Omega}(-\Delta)^s (\varphi(\cdot,0))(\varphi(x,0)-\varphi(x,0))dx\\
=&\|E_s[\varphi(\cdot,0)]\|^2_{\mathcal{X}_{\Sigma_{\mathcal{D}}}^s(\mathcal{C}_{\Omega})}+ \|\varphi-E_s[\varphi(\cdot,0)]\|^2_{\mathcal{X}_{\Sigma_{\mathcal{D}}}^s(\mathcal{C}_{\Omega})},
\end{align*}
which concludes the proof.
\end{proof}
Consider now the following quotient
\begin{equation*}
Q_{\lambda}(w)=\frac{\|w\|_{\mathcal{X}_{\Sigma_{\mathcal{D}}}^s(\mathcal{C}_{\Omega})}^2-\lambda\|u\|_{L^2(\Omega)}^2}{\|u\|_{L^{2_s^*}(\Omega)}^2},
\end{equation*}
where $w=E_s[u]$, and take
\begin{equation}\label{infimo}
S_{\lambda}(\Omega)=\inf\limits_{\substack{w\in \mathcal{X}_{\Sigma_{\mathcal{D}}}^s(\mathcal{C}_{\Omega})\\ w\not\equiv 0}}\big\{Q_{\lambda}(w)\big\}.
\end{equation}
If the constant $S_{\lambda}(\Omega)$ is achieved then problem $(P_{\lambda}^*)$  will have at least one solution, and thus problem $(P_{\lambda})$
 has also at least one solution, as we will see in the proof of Theorem \ref{th_existencia}.
 To study the
 behavior of $Q_{\lambda}(\cdot)$ we introduce the constants $\widetilde{S}(\Sigma_{\mathcal{N}})$ and
 $\widetilde{S}(\Sigma_{\mathcal{D}})$ which are inspired in the works \cite{G} and \cite{ACP} respectively.

\begin{definition}
For $x_0\in\Sigma_{\mathcal{N}}$ we define the function
\begin{align*}
  \Theta_{\lambda} \colon \Sigma_{\mathcal{N}} &\to \mathbb{R}\\
  x_0 &\mapsto \Theta_{\lambda}(x_0),
\end{align*}
by
\begin{equation*}
\Theta_{\lambda}(x_0)=\lim_{\rho\to 0}S_{\lambda}\big(\Omega_{\rho}(x_0)\big),
\end{equation*}
where $\Omega_{\rho}(x_0)=\Omega\cap B_{\rho}(x_0)$ and the
respective infimum in $S_{\lambda}\big(\Omega_{\rho}(x_0)\big)$ is
taken over the set of functions that vanish on
$\Sigma_{\mathcal{D}}^{\rho}=\partial\Omega_{\rho}(x_0)\cap\overline{\Omega}$.\newline
We define the Sobolev constant relative to the Neumann boundary part
as
\begin{equation*}
\widetilde{S}(\Sigma_{\mathcal{N}})=\inf_{x_0\in\Sigma_{\mathcal{N}}}\Theta_{\lambda}(x_0).
\end{equation*}
\end{definition}
This constant plays a major role in the existence issues of problem
$(P_{\lambda})$, similar of that of the Sobolev constant in the
classical Brezis-Nirenberg problem. The next three theorems, which
are going to be proved in Section \ref{sec:4}, will be useful in the
proof of the main result, Theorem \ref{th_existencia}.
\begin{theorem}\label{th1}
If $S_{\lambda}(\Omega)<\widetilde{S}(\Sigma_{\mathcal{N}})$ then the infimum \eqref{infimo} is achieved.
\end{theorem}
As we will see below, the constant
$\widetilde{S}(\Sigma_{\mathcal{N}})$ depends only on the regularity
of the Neumann boundary part, but it is independent of the Dirichlet
boundary part $\Sigma_{\mathcal{D}}$. Since the properties of a
Dirichlet problem are quite different from those of a Neumann
problem, one would expect that this fact is reflected when we move
our boundary conditions, specifically when
$\mathcal{H}_{N-1}(\Sigma_{\mathcal{D}})=\alpha\to0$, see Lemma
\ref{convergence} below. To do so we define the following constant.

\begin{definition}The Sobolev constant relative to the Dirichlet boundary part is defined by
\begin{equation*}
\widetilde{S}(\Sigma_{\mathcal{D}})=\inf_{\substack{u\in
H_{\Sigma_{\mathcal{D}}}^s(\Omega)\\ u\not\equiv
0}}\frac{\|u\|_{H_{\Sigma_{\mathcal{D}}}^s(\Omega)}^2}{\|u\|_{L^{2_s^*}(\Omega)}^2}.
\end{equation*}
\end{definition}
\begin{remark}
As it is noted in the proof of Lemma \ref{lem:traceineq}, the
extension function minimizes the
$\|\cdot\|_{\mathcal{X}_{\Sigma_{\mathcal{D}}}^s(\mathcal{C}_{\Omega})}$
norm along all the functions with the same trace on $\{y=0\}$, thus
we can reformulate the definition of
$\widetilde{S}(\Sigma_{\mathcal{D}})$ as follows,
\begin{equation*}
\widetilde{S}(\Sigma_{\mathcal{D}})=\inf_{\substack{w\in \mathcal{X}_{\Sigma_{\mathcal{D}}}^s(\mathcal{C}_{\Omega})\\
w\not\equiv
0}}\frac{\|w\|_{\mathcal{X}_{\Sigma_{\mathcal{D}}}^s(\mathcal{C}_{\Omega})}^2}{\|w(\cdot,0)\|_{L^{2_s^*}(\Omega)}^2}.
\end{equation*}
\end{remark}
Arguing in a similar way as in \cite[Theorem 2.2]{ACP} we can prove
the following theorem.
\begin{theorem}\label{th_att}
If $\widetilde{S}(\Sigma_{\mathcal{D}})<2^{\frac{-2s}{N}}\kappa_sS(s,N)$ then $\widetilde{S}(\Sigma_{\mathcal{D}})$ is attained.
\end{theorem}
\begin{remarks}\label{rem:cte_Sobolev}
\hfill \break
\begin{enumerate}
\item This result makes the difference between the Dirichlet boundary
condition case and the mixed Dirichlet-Neumann boundary condition
case. Note that, by taking $\lambda=0$ in $(P_{\lambda})$, we have
the critical power problem which, in the Dirichlet case, has no
positive solution under some geometrical assumptions on $\Omega$,
for example, under star-shapeness assumptions on the domain
$\Omega$, see \cite{BrCdPS,Po}, or under some assumptions on the
topology of the domain $\Omega$, see \cite{BahCo}, where  a non-existence result for domains $\Omega$ with trivial
topology is established.
\item In the mixed case, the corresponding Sobolev constant
$\widetilde{S}(\Sigma_{\mathcal{D}})$ can be achieved thanks to
Theorem \ref{th_att}. As we will see, the hypotheses of Theorem
\ref{th_att} can be fulfilled by moving the size of the Dirichlet
boundary part.
\end{enumerate}
\end{remarks}
The next result  is analogous to that of Theorem \ref{th1} for the
constant relative to the Dirichlet part.
\begin{theorem}\label{th_dirichlet}
If $S_{\lambda}(\Omega)<\widetilde{S}(\Sigma_{\mathcal{D}})$ then
$S_{\lambda}(\Omega)$ is attained.
\end{theorem}

\section{Properties of the constants $\widetilde{S}(\Sigma_{\mathcal{N}})$ and
$\widetilde{S}(\Sigma_{\mathcal{D}})$}\label{sec:3}

\begin{proposition}\label{const_neu}
The constant $\widetilde{S}(\Sigma_{\mathcal{N}})$ does not depend on $\lambda$, moreover, if $\Sigma_{\mathcal{N}}$ is a regular $(N-1)$-dimensional
submanifold of $\partial\Omega$, then $\widetilde{S}(\Sigma_{\mathcal{N}})=2^{\frac{-2s}{N}}\kappa_sS(s,N)$.
\end{proposition}
We split the proof into several Lemmas.
\begin{lemma}\label{nolambda}
The constant $\widetilde{S}(\Sigma_{\mathcal{N}})$ does not depend on $\lambda$.
\end{lemma}
\begin{proof}
Note that by the very definition of
$\widetilde{S}(\Sigma_{\mathcal{N}})$ it is enough to prove that
$\Theta_{\lambda}(x_0)$ does not depend on $\lambda$, that is
$\Theta_{\lambda}(x_0)=\Theta(x_0)=\lim\limits_{\rho\to0}S_{0}\big(\Omega_{\rho}(x_0)\big)$.
Since $\lambda\geq0$, then it is immediate that
$\Theta_{\lambda}(x_0)\leq\lim\limits_{\rho\to0}S_{0}\big(\Omega_{\rho}(x_0)\big)$.
On the other hand, using H\"older's inequality and the trace
inequality \eqref{eq:traceineq} jointly, we get
\begin{equation*}
\|\varphi\|_{L^2(\Omega_{\rho})}^2\leq
|\Omega_{\rho}(x_0)|^{\frac{2s}{N}}\|\varphi\|_{L^{2_s^*}(\Omega_{\rho}(x_0))}^2\leq
C|\Omega_{\rho}(x_0)|^{\frac{2s}{N}}\|E_s[\varphi]\|_{\mathcal{X}_{\Sigma_{\mathcal{D}}}^s(\mathcal{C}_{\Omega_{\rho}(x_0)})}^2,
\end{equation*}
thus
\begin{equation*}
\Theta_{\lambda}(x_0)\geq\lim_{\rho\to0}\left(1-\lambda C|\Omega_{\rho}(x_0)|^{\frac{2s}{N}}\right)S_{0}\big(\Omega_{\rho}(x_0)\big).
\end{equation*}
And the result follows.
\end{proof}
Bearing in mind Lemma \ref{nolambda}, to prove the last assertion of
Proposition \ref{const_neu}, we need to estimate
$S_{0}\big(\Omega_{\rho}(x_0)\big)=\inf\left\{Q_{0}(w): w\in
\mathcal{X}_{\Sigma_{\mathcal{D}}}^s
(\mathcal{C}_{\Omega_{\rho}(x_0)})\right\}$. To do so, we use the
family of extremal functions of the Sobolev inequality,
\begin{equation*}
u_{\varepsilon}(x)=\frac{\varepsilon^{\frac{N-2s}{2}}}{(\varepsilon^2+|x|^2)^{\frac{N-2s}{2}}},
\end{equation*}
and its $s$-extension, $w_{\varepsilon}(x)=E_s[u_{\varepsilon}]$,
times a cut-off function as a test function. Note that both
functions $u_{\varepsilon}$ and the Poisson kernel \eqref{poisson}
are self-similar functions,
$u_{\varepsilon}(x)=\varepsilon^{-\frac{N-2s}{2}}u_1(x)$, and
$P_y^s(x)=\frac{1}{y^N}P_1^s\left(\frac{x}{y}\right)$ so the
extension family $w_{\varepsilon}=E_s[u_{\varepsilon}]$ satisfies
\begin{equation}\label{self}
w_{\varepsilon}(x)=\varepsilon^{-\frac{N-2s}{2}}w_{1}\left(\frac{x}{\varepsilon},\frac{y}{\varepsilon}\right).
\end{equation}
Consider a smooth non-increasing cut-off function
$\phi_0(t)\in\mathcal{C}^{\infty}(\mathbb{R}_+)$, satisfying
$\phi_0(t)=1$ for $0\leq t\leq\frac{1}{2}$ and $\phi_0(t)=0$ for
$t\geq1$, and $|\phi_0'(t)|\le C$ for any $t\ge 0$. Assume, without
loss of generality, that $0\in\Omega$, and define, for some $\rho>0$
small enough such that
${B}_{\rho}^+\subseteq{\mathcal{C}}_{\Omega}$, the function
$\phi_{\rho}(x,y)=\phi_0(\frac{r_{xy}}{\rho})$ with
$r_{xy}=|(x,y)|=(|x|^2+y^2)^{\frac{1}{2}}$.

\begin{lemma}\label{lemma_est}
The family $\{\phi_{\rho} w_{\varepsilon}\}$ and its trace on $\{y=0\}$, $\{\phi_{\rho} u_{\varepsilon}\}$, satisfy
\begin{equation}\label{est1}
\|\phi_{\rho}
w_{\varepsilon}\|_{\mathcal{X}_{\Sigma_{\mathcal{D}}}^s(\mathcal{C}_{\Omega})}^2=\|w_{\varepsilon}\|_{\mathcal{X}_{\Sigma_{\mathcal{D}}}^s(\mathcal{C}_{\Omega})}^2+
O\left(\left(\frac{\varepsilon}{\rho}\right)^{N-2s}\right),
\end{equation}
and
\begin{equation}\label{est2}
\int_{\Omega}|\phi_{\rho}
u_{\varepsilon}|^{2_s^*}dx=\|u_{\varepsilon}\|_{L^{2_s^*}(\mathbb{R}^N)}^{2_s^*}+O\left(\left(\frac{\varepsilon}{\rho}\right)^N\right).
\end{equation}
\end{lemma}
The proof of this Lemma is similar to the proof of  \cite[Lemma
3.8]{BCdPS} for the Dirichlet boundary conditions. Note that in the
Dirichlet case it is not necessary to control the role of the radius
of the cut-off function, on the contrary, in the mixed case, by the
very definition of the constant
$\widetilde{S}(\Sigma_{\mathcal{N}})$, a careful analysis of the
role of that radius is needed. Now we estate the following result
proved in \cite[Lemma 3.7]{BCdPS} that will be useful in the proof
of Lemma \ref{lemma_est}.

\begin{lemma} \cite[Lemma 3.7]{BCdPS}
The family $w_{\varepsilon}=w_{\varepsilon,s}=E_s[u_{\varepsilon}]$
satisfies
\begin{equation}\label{mayor}
|\nabla w_{1,s}(x,y)|\leq C w_{1,s-\frac{1}{2}}(x,y),\ \frac{1}{2}<s<1,\ (x,y)\in\mathbb{R}_{+}^{N+1}.
\end{equation}
\end{lemma}

\begin{proof}[Proof of Lemma \ref{lemma_est}]
We start with the proof of \eqref{est2},

\begin{align*}
\int_{\Omega}|\phi_{\rho} u_{\varepsilon}|^{2_s^*}dx&=\int_{\mathbb{R}^N}|\phi_{\rho} u_{\varepsilon}|^{2_s^*}dx\geq \int_{|x|<\frac{\rho}{2}}| u_{\varepsilon}|^{2_s^*}dx\\
&=\|u_{\varepsilon}\|_{L^{2_s^*}(\mathbb{R}^N)}^{2_s^*}-\int_{|x|>\frac{\rho}{2}}|
u_{\varepsilon}|^{2_s^*}dx.
\end{align*}
Observe that
\begin{align*}
\int_{|x|>\frac{\rho}{2}}| u_{\varepsilon}|^{2_s^*}dx&=\varepsilon^{-N}\int_{|x|>\frac{\rho}{2}}\frac{1}{\left(1+\left(\frac{|x|}{\varepsilon}\right)^2\right)^N}dx\\
&=\varepsilon^{-N}\int_{\frac{\rho}{2}}^{\infty}\frac{t^{N-1}}{\left(1+\left(\frac{t}{\varepsilon}\right)^2\right)^N}dt=\int_{\frac{\rho}{2\varepsilon}}^{\infty}\frac{s^{N-1}}{\left(1+s^2\right)^N}ds\\
&\leq \int_{\frac{\rho}{2\varepsilon}}^{\infty}s^{-N-1}ds=\left(\frac{\varepsilon}{\rho}\right)^N,
\end{align*}
so we get
\begin{equation*}
\int_{\Omega}|\phi_{\rho} u_{\varepsilon}|^{2_s^*}dx\geq
\|u_{\varepsilon}\|_{L^{2_s^*}(\mathbb{R}^N)}^{2_s^*}+
O\left(\left(\frac{\varepsilon}{\rho}\right)^N\right).
\end{equation*}
We continue with the proof of \eqref{est1}. The product $\phi_{\rho}w_{\varepsilon}$ satisfies

\begin{align}
\|\phi_{\rho} w_{\varepsilon}\|_{\mathcal{X}_{\Sigma_{\mathcal{D}}}^s(\mathcal{C}_{\Omega})}^2&\leq \|w_{\varepsilon}\|_{\mathcal{X}_{\Sigma_{\mathcal{D}}}^s(\mathcal{C}_{\Omega})}^2\notag\\
&+\kappa_s\int_{\mathcal{C}_{\Omega}}y^{1-2s}|w_{\varepsilon}\nabla\phi_{\rho}|^2dxdy+2\kappa_s\int_{\mathcal{C}_{\Omega}}y^{1-2s}\langle w_{\varepsilon}\nabla\phi_{\rho},\phi_{\rho}\nabla w_{\varepsilon} \rangle dxdy\label{st_lem}.
\end{align}
The first term of the right-hand side in \eqref{st_lem} can be estimated as follows,
\begin{align*}
\int_{\mathcal{C}_{\Omega}}y^{1-2s}|w_{\varepsilon}\nabla\phi_{\rho}|^2dxdy&\leq\frac{C}{\rho^2}\int\limits_{\{\frac{\rho}{2}\leq r_{xy}\leq \rho\}}y^{1-2s}w_{\varepsilon}^2dxdy\\
&\leq\frac{C}{\rho^2}\varepsilon^{N-2s}\!\!\!\!\!\int\limits_{\{\frac{\rho}{2}\leq r_{xy}\leq \rho\}}y^{1-2s}r_{xy}^{-2(N-2s)}dxdy\\
&\leq \frac{C}{\rho^2}\varepsilon^{N-2s}\int_{\frac{\rho}{2}}^{\rho}s^{1+2s-N}ds\\
&=O\left(\left(\frac{\varepsilon}{\rho}\right)^{N-2s}\right),
\end{align*}
since $0\leq u_{\varepsilon}(x)\leq \varepsilon^{\frac{N-2s}{2}}|x|^{-(N-2s)}$ and the extension of the function $K(x)=|x|^{-(N-2s)}$ is $\tilde{K}(x,y)=(|x|^2+y^2)^{-\frac{N-2s}{2}}=r_{xy}^{-(N-2s)}$.
\newline
We end with the estimate of the second term of the right-hand side in \eqref{st_lem}. Applying Cauchy-Schwarz inequality and using \eqref{self} we get,
\begin{equation}\label{coti}
\begin{split}
&\int_{\mathcal{C}_{\Omega}}y^{1-2s}\langle w_{\varepsilon}\nabla\phi_{\rho},\phi_{\rho}\nabla w_{\varepsilon} \rangle dxdy\\
&\leq \frac{C}{\rho}\!\!\!\int\limits_{\{\frac{\rho}{2}\leq r_{xy}\leq \rho\}}y^{1-2s}|w_{\varepsilon}(x,y)\|\nabla w_{\varepsilon}(x,y)|dxdy\\
&\leq \frac{C}{\rho}\varepsilon^{-(N-2s)-1}\!\!\!\!\!\!\!\!\!\int\limits_{\{\frac{\rho}{2}\leq r_{xy}\leq \rho\}}y^{1-2s}|w_{1}\left(\frac{x}{\varepsilon},\frac{y}{\varepsilon}\right)\|\nabla w_{1}\left(\frac{x}{\varepsilon},\frac{y}{\varepsilon}\right)|dxdy\\
&=\frac{C}{\rho}\varepsilon\!\!\!\!\!\int\limits_{\{\frac{\rho}{2\varepsilon}\leq
r_{xy}\leq \frac{\rho}{\varepsilon}\}}y^{1-2s}|w_1(x,y)\|\nabla
w_1(x,y)|dxdy.
\end{split}
\end{equation}
Note that for $(x,y)\in\{\frac{\rho}{2\varepsilon}\leq r_{xy}\leq \frac{\rho}{\varepsilon}\}$ we have
\begin{equation}\label{cotita}
\begin{split}
w_1(x,y)&=\int\limits_{|z|<\frac{\rho}{4\varepsilon}}P_y^s(x-z)u_1(z)dz+\int\limits_{|z|>\frac{\rho}{4\varepsilon}}P_y^s(x-z)u_1(z)dz\\
&\leq C\left(\frac{\varepsilon}{\rho}\right)^{N+2s}y^{2s}\int\limits_{|z|<\frac{\rho}{4\varepsilon}}u_1(z)dz+C\left(\frac{\varepsilon}{\rho}\right)^{N-2s}\int\limits_{|z|>\frac{\rho}{4\varepsilon}}P_y^s(x-z)dz\\
&\leq C\left(\frac{\varepsilon}{\rho}\right)^{N+2s}y^{2s}\int\limits_{|z|<\frac{\rho}{4\varepsilon}}\frac{1}{|z|^{N-2s}}dz+C\left(\frac{\varepsilon}{\rho}\right)^{N-2s}\int_{\mathbb{R}^N}P_y^s(x-z)dz\\
&\leq C\left(\frac{\varepsilon}{\rho}\right)^{N}y^{2s}+C\left(\frac{\varepsilon}{\rho}\right)^{N-2s}\leq C\left(\frac{\varepsilon}{\rho}\right)^{N-2s}.
\end{split}
\end{equation}
Using \eqref{cotita}, \eqref{coti} and \eqref{mayor}, we get
\begin{equation*}
\begin{split}
&\int_{\mathcal{C}_{\Omega}}y^{1-2s}\langle w_{\varepsilon}\nabla\phi_{\rho},\phi_{\rho}\nabla w_{\varepsilon} \rangle dxdy\\
&\leq\frac{C}{\rho}\varepsilon\!\!\!\!\!\!\!\int\limits_{\{\frac{\rho}{2\varepsilon}\leq r_{xy}\leq \frac{\rho}{\varepsilon}\}}y^{1-2s}\left(\frac{\varepsilon}{\rho}\right)^{N-2s}\left(\frac{\varepsilon}{\rho}\right)^{N-2(s-1/2)}dxdy\\
&\leq c\left(\frac{\varepsilon}{\rho}\right)^{2(1+N-2s)}\!\!\!\!\!\!\!\!\!\int\limits_{\{\frac{\rho}{2\varepsilon}\leq r_{xy}\leq \frac{\rho}{\varepsilon}\}}y^{1-2s}dxdy=O\left(\left(\frac{\varepsilon}{\rho}\right)^{N-2s}\right).
\end{split}
\end{equation*}
And the proof is complete.
\end{proof}

\begin{lemma}\label{exactly}
Suppose that $\Sigma_{\mathcal{N}}$ is a regular submanifold of $\partial\Omega$, then given $x_0\in\Sigma_{\mathcal{N}}$ it is satisfied that $\Theta_{\lambda}(x_0)=2^{\frac{-2s}{N}}\kappa_sS(s,N)$.
\end{lemma}
\begin{proof}
From Lemma \ref{nolambda} we know that $\Theta_{\lambda}(x_0)=\Theta(x_0)=\lim\limits_{\rho\to0}S_{0}\big(\Omega_{\rho}(x_0)\big)$, also since $\Sigma_{\mathcal{N}}$ is a regular submanifold of $\partial\Omega$, given $x_0\in\Sigma_{\mathcal{N}}$ we have that,
\begin{equation}\label{proportion}
\lim_{\rho\to0}\frac{|B_{\rho}(x_0)\cap\Omega|}{|B_{\rho}(x_0)|}=\frac{1}{2}.
\end{equation}
On the other hand, since $w_{\varepsilon}$ is a minimizer of $S(s,N)$, we have
\begin{equation*}
S(s,N)=\frac{\displaystyle\int_{\mathbb{R}_+^{N+1}}y^{1-2s}|\nabla
w_{\varepsilon}|^2dxdy}{\|u_{\varepsilon}\|_{L^{2_s^*}(\mathbb{R}^N)}^{2}}.
\end{equation*}
We take now a cut-off function centered at
$x_0\in\Sigma_{\mathcal{N}}$, namely, we take
$\psi_{\rho}(x,y)=\phi_0(\frac{\overline{r}_{xy}}{\rho})$ with
$\overline{r}_{xy}=|(x-x_0,y)|=(|x-x_0|^2+y^2)^{\frac{1}{2}}$. Note
that $\psi_{\rho} u_{\varepsilon}\equiv0$ on
$\partial\Omega_{\rho}\cap\Omega$. Thanks to \eqref{est1} and
\eqref{est2} we can choose $\varepsilon=\rho^{\alpha}$ with
$\alpha>1$ such that

\begin{equation}\label{est1_ro}
\|\phi_{\rho}
w_{\rho}\|_{\mathcal{X}_{\Sigma_{\mathcal{D}}}^s(\mathcal{C}_{\Omega})}^2=\|w_{\rho}\|_{\mathcal{X}_{\Sigma_{\mathcal{D}}}^s(\mathcal{C}_{\Omega})}^2+
O\left(\rho^{(\alpha-1)(N-2s)}\right),
\end{equation}
and
\begin{equation}\label{est2_ro}
\|\phi_{\rho}u_{\rho}\|_{L^{2_s^*}(\Omega)}^{2}=\|u_{\rho}\|_{L^{2_s^*}(\mathbb{R}^N)}^{2}+O\left(\rho^{(\alpha-1)N}\right),
\end{equation}
where $\phi_{\rho}$ is the same cut-off function of Lemma \ref{lemma_est}. Using \eqref{proportion}-\eqref{est2_ro}, we have that
\begin{align*}
\Theta(x_0)&=\lim_{\rho\to0}S_{0}\big(\Omega_{\rho}(x_0)\big)\\
&\leq\lim_{\rho\to0}\frac{\|\psi_{\rho} w_{\rho}\|_{\mathcal{X}_{\Sigma_{\mathcal{D}}}^s(\mathcal{C}_{\Omega})}^2}{\|\psi_{\rho} u_{\rho}\|_{L^{2_s^*}(\Omega)}^2}\\
&=\lim_{\rho\to0}\frac{\frac{1}{2}\|\phi_{\rho} w_{\rho}\|_{\mathcal{X}_{\Sigma_{\mathcal{D}}}^s(\mathcal{C}_{\Omega})}^2}{\frac{1}{2^{\frac{2}{2_s^*}}}\|\phi_{\rho} u_{\rho}\|_{L^{2_s^*}(\Omega)}^2}\\
&=2^{\frac{-2s}{N}}\lim_{\rho\to0}\frac{\kappa_sS(s,N)+O(\rho^{(\alpha-1)(N-2s)})}{1+O(\rho^{(\alpha-1)N})}\\
&=2^{\frac{-2s}{N}}\kappa_sS(s,N).
\end{align*}
Finally, we focus on the proof of inequality
$\Theta(x_0)\geq2^{\frac{-2s}{N}}\kappa_sS(s,N)$. To this end we
assert the following.\newline \textbf{Claim:} For
$x_0\in\Sigma_{\mathcal{N}}$ we have
\begin{equation}\label{eqrt}
\Theta_{\lambda}(x_0)=\Theta(x_0)=\lim\limits_{\rho\to0}S_{0}\big(\Omega_{\rho}(x_0)\big)\geq S_{0}(B_{1}^{+}),
\end{equation}
where $B_{1}^{+}$ is the half ball of radius $1$ centered at $x_0$
with the Neumann boundary part on the flat part of $B_1^+$ and the
Dirichlet boundary part on the closure of the remaining
boundary.\newline To prove the claim, we can argue in a similar way
as in \cite{G}. If \eqref{eqrt} is not true, there exists
$\varepsilon>0,\ r_0>0$, such that for $0<\rho<r_0$ there exists a
function $w_{\rho}\in
\mathcal{X}_{\Sigma_{\mathcal{D}}}^s(\mathcal{C}_{\Omega_{\rho}})$
with $u_\rho=Tr[w_\rho]$ such that

\begin{equation}\label{non}
\frac{\|w_{\rho}\|_{\mathcal{X}_{\Sigma_{\mathcal{D}}}^s(\mathcal{C}_{\Omega_{\rho}})}^2}{\|u_{\rho}\|_{L^{2_s^*}(\Omega_{\rho})}^2}<S_{0}(B_1^{+})-\varepsilon.
\end{equation}
Since $x_0$ is a regular point, there exists a diffeomorfism $T_{\rho}$ between $\Omega_{\rho}$ and $B_{\rho}^{+}$ such that $T_{\rho}(\Sigma_{\mathcal{D}}^{\rho})=\partial B_{\rho}^{+}\cap\partial B(x_0,\rho)$ for $\rho$ small enough. Then the function $v_{\rho}=T_{\rho}(w_{\rho})$ belongs to $\mathcal{X}_{\Sigma_{\mathcal{D}}}^s(\mathcal{C}_{B_{\rho}^+})$ and

\begin{equation*}
\frac{\|v_{\rho}\|_{\mathcal{X}_{\Sigma_{\mathcal{D}}}^s(\mathcal{C}_{B_{\rho}^+})}^2}{\|v_{\rho}(x,0)\|_{L^{2_s^*}(B_{\rho}^+)}^2}
\leq
C_{\rho}\frac{\|w_{\rho}\|_{\mathcal{X}_{\Sigma_{\mathcal{D}}}^s(\mathcal{C}_{\Omega_{\rho}})}^2}{\|u_{\rho}\|_{L^{2_s^*}(\Omega_{\rho})}^2},
\end{equation*}
where $C_{\rho}$ depends on the diffeomorfism $T_{\rho}$ and, by the definition of regular point, it can be chosen in such a way that
$C_{\rho}\to1$ as $\rho\to 0$. Then, for $\rho$ small enough, by \eqref{non} we have

\begin{equation*}
\inf\limits_{\substack{w\in
\mathcal{X}_{\Sigma_{\mathcal{D}}}^s(\mathcal{C}_{B_{\rho}^+})\\
w\not\equiv
0}}\frac{\|w_{\rho}\|_{\mathcal{X}_{\Sigma_{\mathcal{D}}}^s(\mathcal{C}_{B_{\rho}^+})}^2}{\|u_{\rho}\|_{L^{2_s^*}(B_{\rho}^+)}^2}<S_{0}(B_{1}^+),
\end{equation*}
which is a contradiction because, due to the invariance under scaling, we have
\begin{equation*}
\inf\limits_{\substack{w\in
\mathcal{X}_{\Sigma_{\mathcal{D}}}^s(\mathcal{C}_{B_{\rho}^+})\\
w\not\equiv
0}}\frac{\|w_{\rho}\|_{\mathcal{X}_{\Sigma_{\mathcal{D}}}^s(\mathcal{C}_{B_{\rho}^+})}^2}{\|u_{\rho}\|_{L^{2_s^*}(B_{\rho}^+)}^2}=\inf\limits_{\substack{w\in
\mathcal{X}_{\Sigma_{\mathcal{D}}}^s(\mathcal{C}_{B_{1}^+})\\
w\not\equiv 0}}
\frac{\|w_{\rho}\|_{\mathcal{X}_{\Sigma_{\mathcal{D}}}^s(\mathcal{C}_{B_{1}^+})}^2}{\|u_{\rho}\|_{L^{2_s^*}(B_{1}^+)}^2}=S_{0}(B_{1}^+).
\end{equation*}
 Finally, by \eqref{est1}-\eqref{est2} in Lemma \ref{lemma_est} it follows that $S_{0}(B_{1}^+)=2^{\frac{-2s}{N}}\kappa_sS(s,N)$ and hence $\Theta(x_0)\geq2^{\frac{-2s}{N}}\kappa_sS(s,N)$.
\end{proof}

\begin{proof}[Proof of Proposition \ref{const_neu}]
As a consequence of the previous Lemmata we get that if
$\Sigma_{\mathcal{N}}$ is a regular submanifold of $\partial\Omega$
then
$\widetilde{S}(\Sigma_{\mathcal{N}})=2^{\frac{-2s}{N}}\kappa_sS(s,N)$.
\end{proof}

\

We now turn our attention to the Sobolev constant relative to the
Dirichlet part of the boundary
$\widetilde{S}(\Sigma_{\mathcal{D}})$. We give an estimate for
$\widetilde{S}(\Sigma_{\mathcal{D}})$ similar to that of
$\widetilde{S}(\Sigma_{\mathcal{N}})$ in Proposition
\ref{const_neu}.

\begin{proposition}\label{const_dir}
$\widetilde{S}(\Sigma_{\mathcal{D}})\leq 2^{\frac{-2s}{N}}\kappa_sS(s,N)$.
\end{proposition}
\begin{proof}
To obtain this estimate we use the extremal functions of the Sobolev
inequality and proceed in a similar way as in Proposition
\ref{const_neu}. The lower bound in Proposition \ref{const_neu} is
due to the fact that the infimum
$\widetilde{S}(\Sigma_{\mathcal{N}})$ is taken in the set
  $\Omega_{\rho}(x_0)$, on the contrary, for the constant $\widetilde{S}(\Sigma_{\mathcal{D}})$, we do not have such a lower bound by the very
  definition of $\widetilde{S}(\Sigma_{\mathcal{D}})$.
\end{proof}

\section{Proof of main results}\label{sec:4}
\subsection{Proof of Theorem \ref{th_existencia}.$(1)$-$(2)$}\label{subsec:4.1}

\

In this subsection we carry out the proof of Theorems \ref{th1},
\ref{th_att} and \ref{th_dirichlet} which will be useful in the
proof of Theorem \ref{th_existencia}.$(1)$-$(2)$.

We begin with the upper bound of the parameter $\lambda$, i.e.,
statement $(1)$ in Theorem \ref{th_existencia}.
\begin{lemma}\label{eigenf}
Problem $(P_\lambda)$ has no  solution for
$\lambda\geq\lambda_{1,s}$, with $\lambda_{1,s}$ the first
eigenvalue of $(-\Delta)^s$ with mixed boundary condition.
\end{lemma}
\begin{proof}
Assume that $u$ is solution to $(P_\lambda)$ and let $\varphi_1$ be
a positive first eigenfunction of $(-\Delta)^s$. Taking $\varphi_1$
as a test function for $(P_{\lambda})$ we obtain
\begin{equation*}
\lambda_{1,s}\int_{\Omega}u\varphi_1dx=\int_{\Omega}(-\Delta)^{\frac{s}{2}}u(-\Delta)^{\frac{s}{2}}\varphi_1dx=\int_{\Omega}\left(\lambda u+u^{2_s^*-1}\right)\varphi_1dx>\lambda\int_{\Omega}u\varphi_1dx.
\end{equation*}
Therefore, $\lambda<\lambda_{1,s}$.
\end{proof}

\begin{proposition}\label{proplambda}
Assume that $0<\lambda<\lambda_{1,s}$. Then
$S_{\lambda}(\Omega)<2^{\frac{-2s}{N}}\kappa_sS(s,N)$
$=\widetilde{S}(\Sigma_{\mathcal{N}})$.
\end{proposition}
\begin{proof}
We recall the following asymptotic identities given in \cite[Lemma
3.8]{BCdPS},
\begin{equation}\label{est3}
        \|\phi_r u_{\varepsilon}\|_{L^2(\Omega)}^2=\left\{
        \begin{tabular}{lr}
        $C\varepsilon^{2s}+O(\varepsilon^{N-2s})$ & if $N>4s$, \\
        $C\varepsilon^{2s}\log(1/\varepsilon)+O(\varepsilon^{2s})$  & if $N=4s$, \\
        \end{tabular}
        \right.
\end{equation}
for some constant $C>0$, $\varepsilon$ small enough and $\phi_r$ a
cut-off function similar to the one in Lemma \ref{lemma_est}.
Proceeding in a similar way as in Proposition \ref{const_neu}, we
take  a cut-off function centered at a point
$x_0\in\overline{\Sigma}_{\mathcal{N}}$, then using
\eqref{est1}-\eqref{est2} and \eqref{est3} jointly, we have the
following:
\begin{itemize}
\item If $N>4s$,
\begin{align*}
Q_{\lambda}(\phi_r w_{\varepsilon})&\leq2^{\frac{-2s}{N}}\frac{\kappa_sS(s,N)-\lambda C \varepsilon^{2s}\|u_{\varepsilon}\|_{L^{2_s^*}(\Omega)}^{-2}+O(\varepsilon^{N-2s}) }{1+O(\varepsilon^N)}\\
&\leq2^{\frac{-2s}{N}}\kappa_sS(s,N)-\lambda C \varepsilon^{2s}\|u_{\varepsilon}\|_{L^{2_s^*}(\Omega)}^{-2}+O(\varepsilon^{N-2s})\\
&<2^{\frac{-2s}{N}}\kappa_sS(s,N).
\end{align*}
\item If $N=4s$ a similar procedure proves that for
$\varepsilon$ small enough,
\begin{equation*}
Q_{\lambda}(\phi_r
w_{\varepsilon})\leq2^{\frac{-2s}{N}}\kappa_sS(s,N)-\lambda C
\varepsilon^{2s}\log(1/\varepsilon)\|u_{\varepsilon}\|_{L^{2_s^*}(\Omega)}^{-2}+O(\varepsilon^{2s})<2^{\frac{-2s}{N}}\kappa_sS(s,N).
\end{equation*}
\end{itemize}
\end{proof}
Now we enunciate a concentration-compactness result adapted to our
fractional setting with mixed boundary conditions. The proof is a
minor variation of that of the concentration-compactness result in
\cite[Theorem 5.1]{BCdPS}, which is an adaptation to the fractional
setting with Dirichlet boundary conditions of the classical
concentration-compactness technique of P.L. Lions, \cite{Lions}. For
the mixed boundary data case involving the classical Laplace
operator and Caffarelli-Kohn-Nirenberg weights, \cite{CKN}, a
concentration-compactness theorem was proved in \cite{ACP}. First,
we recall the concept of a tight sequence.
\begin{definition}
We say that a sequence $\{y^{1-2s}|\nabla
w_n|^2\}_{n\in\mathbb{N}}\subset L^1(\mathcal{C}_{\Omega})$ is tight
if for any $\eta>0$ there exists $\rho>0$ such that
\begin{equation}\label{tight}
\int_{\{y>\rho\}}\int_{\Omega}y^{1-2s}|\nabla w_n|^2dxdy\leq\eta,\ \forall n\in\mathbb{N}.
\end{equation}
\end{definition}
%Note that the tight condition avoid the evanescence.

\begin{theorem}[Concentration-Compactness]\label{concompact}
Let
$\{w_n\}\subset\mathcal{X}_{\Sigma_{\mathcal{D}}}^s(\mathcal{C}_{\Omega})$
be a weakly convergent sequence to $w$ in
$\mathcal{X}_{\Sigma_{\mathcal{D}}}^s(\mathcal{C}_{\Omega})$ such
that $\{y^{1-2s}|\nabla w_n|^2\}_{n\in\mathbb{N}}$ is tight. Let us
denote $u_n=Tr[w_n]$, $u=Tr[w]$ and let $\mu,\nu$ be two nonnegative
measures such that

\begin{equation}\label{limitcompactness}
y^{1-2s}|\nabla w_n|^2\to\mu,\ \text{and}\ |u_n|^{2_s^*}\to\nu,
\end{equation}
in the sense of measures. Then, there exist an at most countable set
$I$ and points $\{x_i\}_{i\in I}\subset\overline{\Omega}$ such that
\begin{enumerate}
\item $\nu=|u|^{2_s^*}+\sum\limits_{i\in I}\nu_{i}\delta_{x_{i}},\ \nu_{i}>0,$
\item $\mu=y^{1-2s}|\nabla w|^2+\sum\limits_{i\in I}\mu_{i}\delta_{x_{i}},\ \mu_{i}>0,$
\item $\mu_i\geq \widetilde{S}(\Sigma_{\mathcal{D}})\nu_i^{\frac{2}{2_s^*}}$.
\end{enumerate}
\end{theorem}
Using Theorem \ref{concompact} we prove the next result that is
analogous to \cite[Theorem 2.2]{LPT}.
\begin{theorem}\label{CCA}
Let $w_m$ be a minimizing sequence of $S_{\lambda}(\Omega)$. Then
either $w_m$ is re\-la\-ti\-vely compact or the weak limit, $w\equiv
0$. Even more, in the latter case there exist a subsequence $w_{m}$
and a point $x_0\in\overline{\Sigma}_{\mathcal{N}}$ such that
\begin{equation}\label{acumulacion}
y^{1-2s}|\nabla w_m|^2\to S_{\lambda}(\Omega)\delta_{x_0},\ \text{and}\ |u_m|^{2_s^*}\to\delta_{x_0},
\end{equation}
with $u_m=Tr[w_m]$.
\end{theorem}
\begin{proof}
Since $0\le\lambda<\lambda_{1,s}$ it follows that
$0<S_{\lambda}(\Omega)\leq\widetilde{S}(\Sigma_{\mathcal{D}})$. We
distinguish two cases, depending upon if
$S_{\lambda}(\Omega)<\widetilde{S}(\Sigma_{\mathcal{D}})$ or
$S_{\lambda}(\Omega)=\widetilde{S}(\Sigma_{\mathcal{D}})$:

\

\textbf{(1)}
$S_{\lambda}(\Omega)<\widetilde{S}(\Sigma_{\mathcal{D}})$. In this
case we can argue in a similar way as in \cite[Prop. 4.2]{BCdPS}
which in turn is based on the technique of Brezis-Nirenberg.

Let $\{w_{m}\}\subset
\mathcal{X}_{\Sigma_{\mathcal{D}}}^s(\mathcal{C}_{\Omega})$ be a
minimizing sequence of $S_{\lambda}(\Omega)$, and suppose without
loss of generality that $w_m\geq0$ and
$\|w_{m}(\cdot,0)\|_{L^{2_s^*}(\Omega)}=1$. Clearly, this implies
that
\begin{equation}\label{eq:cota_M}
\|w_m\|_{\mathcal{X}_{\Sigma_{\mathcal{D}}}^s(\mathcal{C}_{\Omega})}\leq
M,
\end{equation} then, there exists a subsequence (denoted also by
$\{w_m\}$) verifying,
\begin{align*}
w_m&\rightharpoonup w\ \text{weakly in}\ \mathcal{X}_{\Sigma_{\mathcal{D}}}^s(\mathcal{C}_{\Omega}),\\
w_m(\cdot,0)&\to w(\cdot,0)\ \text{strongly in}\ L^{q}(\Omega),\ 1\leq q<2_s^*,\\
w_m(\cdot,0)&\to w(\cdot,0)\ \text{a.e. in}\ \Omega.
\end{align*}
Using the weak convergence we get
\begin{align*}
\|w_m\|_{\mathcal{X}_{\Sigma_{\mathcal{D}}}^s(\mathcal{C}_{\Omega})}^2&=\|w_m-w\|_{\mathcal{X}_{\Sigma_{\mathcal{D}}}^s(\mathcal{C}_{\Omega})}^2\\
&\,\quad +\|w\|_{\mathcal{X}_{\Sigma_{\mathcal{D}}}^s(\mathcal{C}_{\Omega})}^2+2\kappa_s\int_{\mathcal{C}_{\Omega}}y^{1-2s}\langle\nabla w,\nabla w_m-\nabla w\rangle dxdy\\
&=\|w_m-w\|_{\mathcal{X}_{\Sigma_{\mathcal{D}}}^s(\mathcal{C}_{\Omega})}^2+\|w\|_{\mathcal{X}_{\Sigma_{\mathcal{D}}}^s(\mathcal{C}_{\Omega})}^2+o(1).
\end{align*}
Hence,
\begin{align*}
Q_{\lambda}(w_m)&=\|w_m\|_{\mathcal{X}_{\Sigma_{\mathcal{D}}}^s(\mathcal{C}_{\Omega})}^2-\lambda\|w_m(\cdot,0)\|_{L^2(\Omega)}^2\\
&=\|w_m-w\|_{\mathcal{X}_{\Sigma_{\mathcal{D}}}^s(\mathcal{C}_{\Omega})}^2+\|w\|_{\mathcal{X}_{\Sigma_{\mathcal{D}}}^s(\mathcal{C}_{\Omega})}^2-\lambda\|w_m(\cdot,0)\|_{L^2(\Omega)}^2+o(1)\\
&\geq
\widetilde{S}(\Sigma_{\mathcal{D}})\|w_{m}(\cdot,0)-w(\cdot,0)\|_{L^{2_s^*}(\Omega)}^2+S_{\lambda}(\Omega)\|w(\cdot,0)\|_{L^{2_s^*}(\Omega)}^2+o(1).
\end{align*}
Thus, because of the normalization  $\|w_m (\cdot,
0)\|_{L^{2^*_s}}=1$, it follows
\begin{align*}
Q_{\lambda}(w_m)&\geq
(\widetilde{S}(\Sigma_{\mathcal{D}})-S_{\lambda}(\Omega))\|w_{m}(\cdot,0)-w(\cdot,0)\|_{L^{2_s^*}(\Omega)}^2+S_{\lambda}(\Omega)+
o(1).
\end{align*}
Since $\{w_m\}$ is a minimizing sequence of $S_{\lambda}(\Omega)$, we obtain
\begin{equation*}
o(1)+S_{\lambda}(\Omega)\geq
(\widetilde{S}(\Sigma_{\mathcal{D}})-S_{\lambda}(\Omega))\|w_{m}(\cdot,0)-w(\cdot,0)\|_{L^{2_s^*}(\Omega)}^2+S_{\lambda}(\Omega)+o(1).
\end{equation*}
Finally, using that
$S_{\lambda}(\Omega)<\widetilde{S}(\Sigma_{\mathcal{D}})$ it follows
\begin{equation*}
w_m(\cdot,0)\to w(\cdot,0)\ \text{in}\ L^{2_s^*}(\Omega).
\end{equation*}
By a standard lower semi-continuity argument, $w$ is a minimizer for $Q_{\lambda}(\cdot)$, so we get that the sequence is relatively compact.\newline

\textbf{(2)}
$S_{\lambda}(\Omega)=\widetilde{S}(\Sigma_{\mathcal{D}})$. Let us
consider $\{w_m\}\subset
\mathcal{X}_{\Sigma_{\mathcal{D}}}^s(\mathcal{C}_{\Omega})$ as in
the previous case. Thus $\{w_m\}$ is also a minimizing sequence for
$\widetilde{S}(\Sigma_{\mathcal{D}})$ and we proceed in a similar
way as in \cite[Theorem 2.2]{LPT}. Using Theorem \ref{concompact},
we get that  either $\{w_{m}\}$ is relatively compact or the weak
limit $w\equiv0$.
\newline
In the first case, $w\not\equiv 0$, by Theorem \ref{concompact} we
have
\begin{equation*}
\|u\|_{L^{2_s^*}(\Omega)}^2
\widetilde{S}(\Sigma_{\mathcal{D}})=\int_{\mathcal{C}_{\Omega}}d\mu
=\int_{\mathcal{C}_{\Omega}}y^{1-2s}|\nabla
w|^2dxdy+\widetilde{S}(\Sigma_{\mathcal{D}})\sum_{i\in
I}\nu_{i}^{\frac{2}{2_s^*}},
\end{equation*}
as well as
\begin{equation*}
\int_{\Omega}d\nu=\int_{\Omega}|u|^{2_s^*}dx+\sum_{i\in
I}\nu_{i}\geq\int_{\Omega}|u|^{2_s^*}dx.
\end{equation*}
By the two expressions  above, and using that
$\|u\|_{L^{2_s^*}(\Omega)}=1$ we get,
\begin{align}\label{cuentita}
\left(1-\sum_{i\in I}\nu_{i}\right)^{\frac{2}{2_s^*}}&\leq\frac{1}{\widetilde{S}(\Sigma_{\mathcal{D}})}\int_{\mathcal{C}_{\Omega}}y^{1-2s}|\nabla w|^2dxdy \\
&\leq\frac{1}{\widetilde{S}(\Sigma_{\mathcal{D}})}\left( \widetilde{S}(\Sigma_{\mathcal{D}})-\widetilde{S}(\Sigma_{\mathcal{D}})\sum_{i\in I}\nu_{i}^{\frac{2}{2_s^*}}\right)\nonumber \\
&= 1-\sum_{i\in I}\nu_{i}^{\frac{2}{2_s^*}}\nonumber ,
\end{align}
hence, $ {\nu}_i\leq 1$ $\forall i\in I.$ And therefore,  by
\eqref{cuentita} the only possibility is $\nu_i=0$ for all $i\in I$.
This leads to
\begin{equation*}
\int_{\Omega}|u_m|^{2_s^*}dx\to\int_{\Omega}|u|^{2_s^*}dx,
\end{equation*}
from which we deduce that $u_m$ (and thus $w_m=E_s[u_m]$) is
relatively compact.

Now we consider the case $w\equiv 0$ (and thus $u\equiv0$). In this
case by Theorem \ref{concompact} and \eqref{cuentita} we get
\begin{equation*}
\sum_{i\in I}\nu_i=1,\ \mbox{and}\ \sum_{i\in
I}\nu_i^{\frac{2}{2_s^*}}\leq 1,
\end{equation*}
then we infer that $I$ must be a singleton, i.e.,
\begin{equation*}
\nu=\delta_{x_0}\ \ \ \ \ \mbox{and}\ \ \ \ \ \mu=\widetilde{S}(\Sigma_{\mathcal{D}})\delta_{x_0}=S_{\lambda}(\Omega)\delta_{x_0},
\end{equation*}
with $x_0\in\overline{\Omega}$.

\

To show that $x_0\in\overline{\Sigma}_{\mathcal{N}}$ we argue by
contradiction. If $x_0\in\Omega\cup\Sigma_{\mathcal{D}}$, we set
$\overline{\phi}_r(x,y)$ as a cut-off function centered at
$x_0\in\Omega$, and define the sequence
\begin{equation*}
w_{m,r}=w_m\overline{\phi}_r(x,y)
\end{equation*}
and the traces sequence $\{u_{m,r}\}=\{Tr[w_{m,r}]\}$. Then for all $r>0$

\begin{equation}\label{contradiction}
\lim_{m\to\infty}\frac{\displaystyle\int_{\mathcal{C}_{\Omega}}y^{1-2s}|\nabla
w_{m,r}|^2dxdy}{\|u_{m,r}\|_{L^{2_s^*}(\Omega)}^2}=\widetilde{S}(\Sigma_{\mathcal{D}}).
\end{equation}
Note that for $r$ sufficiently small, the sequence $\{w_{m,r}\}$
belongs to $\mathcal{X}_{0}^s(\mathcal{C}_{\Omega})$, then for any
$m\in\mathbb{N}$, by Proposition \ref{const_dir},
\begin{align*}
\lim_{r\to0}\frac{\int_{\mathcal{C}_{\Omega}}y^{1-2s}|\nabla
w_{m,r}|^2dxdy}{\|u_{m,r}\|_{L^{2_s^*}(\Omega)}^2}&\geq
\inf_{\substack{v\in\mathcal{X}_{0}^s(\mathcal{C}_{\Omega})\\ v\not\equiv 0}}\frac{\int_{\mathcal{C}_{\Omega}}y^{1-2s}|\nabla v|^2dxdy}{\|v(x,0)\|_{L^{2_s^*}(\Omega)}^2}\\
&=\kappa_sS(s,N)\\
&>2^{\frac{-2s}{N}}\kappa_sS(s,N)\\
&\geq\widetilde{S}(\Sigma_{\mathcal{D}}),
\end{align*}
and we reach a contradiction with \eqref{contradiction}. Therefore,
$x_0\in\partial\Omega$. If $x_0\in\mathring{\Sigma}_{\mathcal{D}}$
arguing as before we reach the same contradiction. As a consequence,
$x_0\in\overline{\Sigma}_{\mathcal{N}}$.

It only remains to prove the tightness condition \eqref{tight} for
the minimizing sequence $\{w_m\}\subset
\mathcal{X}_{\Sigma_{\mathcal{D}}}^s(\mathcal{C}_{\Omega})$, i.e.,
there is no evanescence. Since $\{w_m\}$ is a minimizing sequence of
$S_{\lambda}(\Omega)$ then $\{w_m\}$ or a multiple will converge to
a critical point of the functional \eqref{extensionfunctional}. Let
$\{\tilde{w}_{m}\}$ be such a sequence, then
\begin{equation}\label{minima}
J(\tilde{w}_{m})\to c,\ \mbox{and}\ J'(\tilde{w}_{m})\to 0.
\end{equation}
We proceed now as in \cite[Lemma 3.6]{BCdPS} which is based on ideas
contained in \cite{AAP}. By contradiction, suppose that there exists
$\eta_0>0$, and $m_0\in \mathbb{N}$ such that for any $\rho>0$ one
has, up to a subsequence,

\begin{equation}\label{pp}
\int_{\{y>\rho\}}\int_{\Omega}y^{1-2s}|\nabla
\tilde{w}_{m}|^2dxdy>\eta_0,\ \forall m\ge m_0.
\end{equation}
Fix $\varepsilon>0$ (to be determined) and let $r>0$ be such that
\begin{equation*}
 \int_{\{y>r\}}\int_{\Omega}y^{1-2s}|\nabla \tilde{w}|^2dxdy<\varepsilon.
\end{equation*}
Let $j=\left[\frac{M}{\kappa_s\varepsilon}\right]$ be the integer
part with $M$ the constant in \eqref{eq:cota_M} and
$I_k=\{y\in\mathbb{R^+}:r+k\leq y\leq r+k+1\}$, $k=0,1,\ldots, j$.
Then
\begin{equation*}
\sum_{k=0}^{j}\int_{I_k}\int_{\Omega}y^{1-2s}|\nabla \tilde{w}_{m}|^2dxdy\leq \int_{\mathcal{C}_{\Omega}}y^{1-2s}|\nabla \tilde{w}_{m}|^2dxdy\leq\frac{M}{\kappa_s}<\varepsilon(j+1).
\end{equation*}
Then, there exists $k_0\in\{0,\ldots, j\}$ such that, up to a subsequence,
\begin{equation}\label{peque}
\int_{I_{k_0}}\int_{\Omega}y^{1-2s}|\nabla
\tilde{w}_{m}|^2dxdy\leq\varepsilon,\ \forall m\ge m_0.
\end{equation}
We set now a regular cut-off function
\begin{equation*}
        \chi(y)=\left\{
        \begin{tabular}{lcl}
        $0$ & & if $y\leq r+k_0$, \\
        $1$ & & if $y>r+k_0+1$,
        \end{tabular}
        \right.
\end{equation*}
and we define $v_m(x,y)=\chi(y)\tilde{w}_{m}(x,y)$. Then, since
$v_m(x,0)=0$, it follows that

\begin{align*}
|\langle J'(\tilde{w}_{m})-J'(v_m),v_m\rangle|&=\kappa_s\int_{\mathcal{C}_{\Omega}}y^{1-2s}\langle \nabla(\tilde{w}_{m}-v_m),\nabla v_m\rangle dxdy\\
&=\kappa_s\int_{I_{k_0}}\int_{\Omega}y^{1-2s}\langle\nabla(\tilde{w}_{m}-v_m),\nabla v_m\rangle dxdy.
\end{align*}
Moreover, by the Cauchy-Schwarz inequality, \eqref{peque} and the
compact inclusion of the space
$H^1\left(I_{k_0}\times\Omega,y^{1-2s}dxdy\right)$ into
$L^2\left(I_{k_0}\times\Omega,y^{1-2s}dxdy\right)$, it follows that

\begin{align*}
&|\langle J'(\tilde{w}_{m})-J'(v_m),v_m\rangle|\\
&\leq\kappa_s\left(\int_{I_{k_0}}\int_{\Omega}y^{1-2s}|\nabla(\tilde{w}_{m}-v_m)|^2 dxdy\right)^{1/2}\left(\int_{I_{k_0}}\int_{\Omega}y^{1-2s}|\nabla v_m|^2 dxdy\right)^{1/2}\\
&\leq C\kappa_s \varepsilon.
\end{align*}
Finally, by \eqref{minima},

\begin{equation*}
|\langle J'(v_m),v_m\rangle|\leq C\kappa_s \varepsilon+o(1),
\end{equation*}
thus, for $m$ big enough
\begin{equation*}
\int\limits_{\{y>r+k_0+1\}}\int_{\Omega}y^{1-2s}|\nabla w_m|^2dxdy\leq\int_{\mathcal{C}_{\Omega}}\int_{\Omega}y^{1-2s}|\nabla v_m|^2dxdy\leq\frac{\langle J'(v_m),v_m\rangle}{\kappa_s}\leq C\varepsilon,
\end{equation*}
which contradicts \eqref{pp}. Then, the proof of Theorem \ref{CCA} is complete.
\end{proof}
\begin{remark}
Note that the proof of Theorem \ref{th_dirichlet} was done in the first part of the proof of Theorem \ref{CCA}.
\end{remark}
Now we prove Theorems \ref{th1}, \ref{th_att}.
\begin{proof}[Proof of Theorem \ref{th_att}]
Let $\{w_m\}\subset
\mathcal{X}_{\Sigma_{\mathcal{D}}}^s(\mathcal{C}_{\Omega})$ be a
minimizing sequence of $\widetilde{S}(\Sigma_{\mathcal{D}})$ and $w$
its weak limit. By Theorem \ref{CCA}, $\{w_m\}$ is relatively
compact, and consequently the infimum is achieved, or $w\equiv0$ and
\begin{equation*}
y^{1-2s}|\nabla w_n|^2\to \mu\delta_{x_0},\ \text{and}\
|u_n|^{2_s^*}\to\nu\delta_{x_0},
\end{equation*}
with $x_0\in\overline{\Sigma}_{\mathcal{N}}$. Indeed, we can assume,
without loss of generality, that
$\mu=\widetilde{S}(\Sigma_{\mathcal{D}})$ and $\nu=1$. With the same
notation as in the proof of Theorem \ref{CCA}, we consider the
functions
\begin{equation}\label{eq:barra}
w_{m,r}=w_m\overline{\phi}_r(x,y)
\end{equation}
with $\overline{\phi}_r(x,y)$ a smooth cut-off function centered at
$x_0\in\overline{\Sigma}_{\mathcal{N}}$. Clearly, \eqref{eq:barra}
satisfies \eqref{contradiction}. Since $\Sigma_{\mathcal{N}}$ is
smooth, for $r$ small enough, the sequence $\{u_{m,r}\}\subset
H_{\Sigma_{\mathcal{D}}^{r}}^s(\Omega_r)$, or equivalently, the
sequence
$\{w_{m,r}\}\subset\mathcal{X}_{\Sigma_{\mathcal{D}}^{r}}^s(\mathcal{C}_{\Omega_r})$
thus,  by Proposition \ref{const_dir},
\begin{equation*}
\lim_{r\to0}\frac{\displaystyle\int_{\mathcal{C}_{\Omega}}y^{1-2s}|\nabla
w_{m,r}|^2dxdy}{\|u_{m,r}\|_{L^{2_s^*}(\Omega)}^2}\geq
2^{\frac{-2s}{N}}\kappa_sS(s,N)>\widetilde{S}(\Sigma_{\mathcal{D}}),
\end{equation*}
which contradicts \eqref{contradiction}. Then the only possibility is that $\{w_m\}$ is relatively compact, which proves the assertion.

\end{proof}

\begin{proof}[Proof of Theorem \ref{th1}]
Let $\{w_m\}\subset
\mathcal{X}_{\Sigma_{\mathcal{D}}}^s(\mathcal{C}_{\Omega})$ be a
minimizing sequence for $S_{\lambda}(\Omega)$ and $w$ its weak
limit. Thus, either $\{w_m\}$ is relatively compact and consequently
the infimum is achieved or by Theorem \ref{CCA}, \eqref{acumulacion}
holds up to a subsequence. For that sequence we consider the
functions $w_{m,r}=w_m\overline{\phi}_r(x,y),$ with
$\overline{\phi}_r(x,y)$ a smooth cut-off function centered at
$x_0\in\overline{\Sigma}_{\mathcal{N}}$ as in \eqref{eq:barra}.On
the one hand, $\{w_{m,r}\}$ and its trace $\{u_{m,r}\}$ satisfy

\begin{equation}\label{achieved}
\frac{\|w_{m,r}\|_{\mathcal{X}_{\Sigma_{\mathcal{D}}}^s(\mathcal{C}_{\Omega})}^2-\lambda\|u_{m,r}\|_{L^2(\Omega)}^2}{\|u_{m,r}\|_{L^{2_s^*}(\Omega)}^2}\to
S_{\lambda}(\Omega),\quad\mbox{as}\quad m\to\infty,
\end{equation}
 for any $r>0$. On the other, by the definition of $\widetilde{S}(\Sigma_{\mathcal{N}})$ we have

\begin{equation*}
\lim_{r\to0}
\frac{\|w_{m,r}\|_{\mathcal{X}_{\Sigma_{\mathcal{D}}}^s(\mathcal{C}_{\Omega})}^2-\lambda\|u_{m,r}\|_{L^2(\Omega)}^2}{\|u_{m,r}\|_{L^{2_s^*}(\Omega)}^2}
\geq\widetilde{S}(\Sigma_{\mathcal{N}}),
\end{equation*}
which contradicts \eqref{achieved} since we are supposing
$S_{\lambda}(\Omega)<\widetilde{S}(\Sigma_{\mathcal{D}})$. Hence
$\{w_{m}\}$ is relatively compact.
\end{proof}

\begin{proof}[Proof of Theorem \ref{th_existencia}-$(2)$]
By Theorem \ref{th1}, it follows inmediatly the existence of a
solution to problem $(P_{\lambda})$ whenever we have
$S_{\lambda}(\Omega)<\widetilde{S}(\Sigma_{\mathcal{N}})$, which is
guaranteed by Proposition \ref{proplambda} if
$0<\lambda<\lambda_{1,s}$. Also, there exists a solution when
$S_{\lambda}(\Omega)<\widetilde{S}(\Sigma_{\mathcal{D}})$ by Theorem
\ref{th_dirichlet}.\newline Specifically, by Theorem \ref{th1} and
Proposition \ref{proplambda}, if $0<\lambda<\lambda_{1,s}$ there
exists a minimizer function $\tilde{w}$ with
$\tilde{u}=Tr[\tilde{w}]$ satisfying
\begin{equation*}
\|\tilde{w}\|_{\mathcal{X}_{\Sigma_{\mathcal{D}}}^s(\mathcal{C}_{\Omega})}^2-\lambda\|\tilde{u}\|_{L^2(\Omega)}^2=S_{\lambda}(\Omega)\|\tilde{u}\|_{L^{2_s^*}(\Omega)}^2.
\end{equation*}
Taking $w=\tilde{w}/\|\tilde{u}\|_{L^{2_s^*}(\Omega)}^2$ and its
trace $u=\tilde{u}/\|\tilde{u}\|_{L^{2_s^*}(\Omega)}^2$,
\begin{equation}\label{ult}
\|w\|_{\mathcal{X}_{\Sigma_{\mathcal{D}}}^s(\mathcal{C}_{\Omega})}^2-\lambda\|u\|_{L^2(\Omega)}^2=S_{\lambda}(\Omega).
\end{equation}
Thus $w$ is a minimizer of $S_{\lambda}(\Omega)$ constrained to the
sphere $\|u\|_{L^{2_s^*}(\Omega)}=1$. Without loss of generality we
can assume $w\geq0$, otherwise we take $|w|$ instead. Or
equivalently, $w$ is a critical point of the functional
$Q_{\lambda}$ constrained to $\|u\|_{L^{2_s^*}(\Omega)}^2=1$, then
thanks to \eqref{norma1} and \eqref{norma2}, such a critical point
is a non-negative solution to equation
\begin{equation*}
(-\Delta)^su-\lambda u=\tau u^{2_s^*-1}\,\,\, \mbox{in} \,\,\,
\Omega,
\end{equation*}
where $\tau\in\mathbb{R}$ is a Lagrange multiplier. Moreover
$\tau=S_{\lambda}(\Omega)>0$ since $\lambda<\lambda_{1,s}$. Thus, it
follows that defining $v=ku$, it is a non-negative solution to the
equation in $(P_{\lambda})$ for
$k=\left(S_{\lambda}(\Omega)\right)^{\frac{1}{2_s^*-2}}$. Even more,
by the maximum principle, $v>0$ in $\Omega$, proving that it is a
solution to $(P_{\lambda})$.
\end{proof}
\begin{remark}
By Proposition \ref{proplambda}, if  $0<\lambda<\lambda_{1,s}$ then
the Neumann constant  satisfies
$S_{\lambda}(\Omega)<\widetilde{S}(\Sigma_{\mathcal{N}})$, while for
the Dirichlet constant $\widetilde{S}(\Sigma_{\mathcal{D}})$, we
have not such a result because we do not know an explicit expression
of the corresponding minimizers and hence, we can not provide
estimates similar to those of Lemma \ref{lemma_est}.
\end{remark}
To complete the proof of Theorem \ref{th_existencia} it only remains
to prove statement (3) in Theorem \ref{th_existencia}. This will be
done in the next subsection.

\subsection{Moving the boundary conditions. Proof of Theorem \ref{th_existencia}-$(3)$}\label{subsec:4.2}

\

Let us consider the following eigenvalue problem
\begin{equation*}
        \left\{
        \begin{tabular}{lcl}
        $(-\Delta)^s u=\lambda_{1,s}(\alpha)u$ & &in $\Omega\subset \mathbb{R}^{N}$,\\
        $u=0$& &on $\Sigma_{\mathcal{D}}(\alpha)$,\\
        $\frac{\partial u}{\partial\nu}=0$  & &on $\Sigma_{\mathcal{N}}(\alpha)$,\\
        \end{tabular}
        \right.
        \tag{$EP_{\alpha}$}
\end{equation*}
with the following hypotheses:
\begin{enumerate}
\item[$B_1:$] $\Omega\subset\mathbb{R}^N$ is a regular bounded domain.
\item[$B_2:$] $\Sigma_{\mathcal{D}}(\alpha)$ and $\Sigma_{\mathcal{N}}(\alpha)$ are smooth $(N-1)$-dimensional submanifolds of $\partial\Omega$
such that
$\Sigma_{\mathcal{D}}(\alpha)\cup\Sigma_{\mathcal{N}}(\alpha)=\partial\Omega$,
$\Sigma_{\mathcal{D}}(\alpha)\cap\Sigma_{\mathcal{N}}(\alpha)=\emptyset$,
and the interphase
$\Gamma(\alpha)=\Sigma_{\mathcal{D}}(\alpha)\cap\overline{\Sigma}_{\mathcal{N}}(\alpha)$
is a $(N-2)$-dimensional submanifold.
\item[$B_3:$] $\mathcal{H}_{N-1}(\Sigma_{\mathcal{D}}(\alpha))=\alpha$, $\Sigma_{\mathcal{D}}(\alpha_1)\subseteq\Sigma_{\mathcal{D}}(\alpha_2)$ for any $0<\alpha_1\leq\alpha_2<\mathcal{H}_{N-1}(\partial\Omega)$.
\end{enumerate}
Following \cite[Lemma 4.1]{ColP} we have the next result.
\begin{lemma}\label{convergence}
Let $u_{\alpha}$ be a positive solution to problem $(EP_{\alpha})$
and suppose hypotheses $B_1$-$B_3$. Then we obtain,
\begin{equation*}
\lambda_{1,s}(\alpha)\to 0,\ \text{as}\ \alpha\to0.
\end{equation*}
\end{lemma}
\begin{proof}
By the definition of the fractional operator $(-\Delta)^s$, we have
that the eigenvalue $\lambda_{1,s}(\alpha)=\lambda_{1,1}^s(\alpha)$.

By \cite[Lemma 4.1]{ColP}, we have that $\lambda_{1,1}(\alpha)\to 0 $ as
$\mathcal{H}_{N-1}(\Sigma_{\mathcal{D}}(\alpha))=\alpha\to 0$. Then
the result follows.
\end{proof}

The next proposition is the analogous to \cite[Proposition
2.1]{ACP} for our fractional setting.

\begin{proposition}\label{att}
Let $\Omega\subset\mathbb{R}^N$ be a smooth bounded domain. Given a
family
$\{\Sigma_{\mathcal{D}}(\alpha):0<\alpha<\mathcal{H}_{N-1}(\partial\Omega)\}$
satisfying hypotheses $B_1$-$B_3$, there exists a positive constant
$\alpha_0$ such that for any
$\alpha<\alpha_0$, $\widetilde{S}(\Sigma_{\mathcal{D}}(\alpha))$ is attained.
\end{proposition}
\begin{proof}
 We only have to check that hypotheses of Theorem
\ref{th_att} are satisfied. To do so,
 we use the H\"older inequality
together with Lemma \ref{convergence} as follows. By H\"older's
inequality,
\begin{equation}\label{eq:convergencia}
\begin{array}{rcl}
\widetilde{S}(\Sigma_{\mathcal{D}}(\alpha))&=&\inf_{\substack{w\in \mathcal{X}_{\Sigma_{\mathcal{D}}\!(\alpha\!)}^s(\mathcal{C}_{\Omega})\\ w\not\equiv 0}}\frac{\|w\|_{\mathcal{X}_{\Sigma_{\mathcal{D}}}^s(\mathcal{C}_{\Omega})}^2}{\|w(\cdot,0)\|_{L^{2_s^*}(\Omega)}^2}\\
&\leq & |\Omega|^{\frac{2s}{N}} \inf_{\substack{w\in \mathcal{X}_{\Sigma_{\mathcal{D}}\!(\alpha\!)}^s(\mathcal{C}_{\Omega})\\ w\not\equiv 0}}\frac{\|w\|_{\mathcal{X}_{\Sigma_{\mathcal{D}}}^s(\mathcal{C}_{\Omega})}^2}{\|w(\cdot,0)\|_{L^{2}(\Omega)}^2}\\
&= & |\Omega|^{\frac{2s}{N}} \lambda_{1,s}(\alpha).
\end{array}\end{equation}
Applying Lemma \ref{convergence} into \eqref{eq:convergencia}, we have that there exists $\alpha_0>0$ such that $\widetilde{S}(\Sigma_{\mathcal{D}}(\alpha))<2^{\frac{-2s}{N}}\kappa_sS(s,N)$ for any $\alpha<\alpha_0$. Hence, by Theorem \ref{th_att} the result follows.
\end{proof}
We complete now the proof of Theorem \ref{th_existencia}.
\begin{proof}[Proof of Theorem \ref{th_existencia}-$(3)$] Since $S_{\lambda}(\Omega)=\widetilde{S}(\Sigma_{\mathcal{D}})$ for $\lambda=0$,
 the existence of solution to problem $(P_0)$ is equivalent to the attainability of $\widetilde{S}(\Sigma_{\mathcal{D}})$. Thus, letting $\alpha$ sufficiently small, by Proposition \ref{att} there exists a minimizer function $\tilde{w}$ with $\tilde{u}=Tr[\tilde{w}]$ satisfying
\begin{equation*}
\|\tilde{w}\|_{\mathcal{X}_{\Sigma_{\mathcal{D}}}^s(\mathcal{C}_{\Omega})}^2=\widetilde{S}(\Sigma_{\mathcal{D}})\|\tilde{u}\|_{L^{2_s^*}(\Omega)}^2,
\end{equation*}
and we are done.
\end{proof}
%%%%%%%%%%%%%%%%%%%%%%%%%%%%%%%%%%%%%%%%%%%%%%%%%%%%%%%%%%%%%%%%%%%%%%%%
%%%%%%%%%%%%%%%%%%%%%%%%%%%%%%%%%%%%%%%%%%%%%%%%%%%%%%%%%%%%%%%%%%%%%%%%
\section{A nonexistence result: Pohozaev-type identity}\label{sec:6}
This last part deals with a non-existence result relying on a
Pohozaev-type identity. Notice that by Theorem
\ref{th_existencia}-$(3)$ we have the existence of solution to the
following critical problem,
\begin{equation}\label{prcrt}
        \left\{
        \begin{tabular}{lcl}
        $(-\Delta)^su=u^{2_s^*-1}$ & &in $\Omega\subset \mathbb{R}^{N}$, \\
        $u>0$& & in $\Omega$,\\
        $B(u)=0$  & &on $\partial\Omega=\Sigma_{\mathcal{D}}\cup\Sigma_{\mathcal{N}}$, \\
        \end{tabular}
        \right.
\end{equation}
provided $\alpha=\mathcal{H}_{N-1}(\Sigma_{\mathcal{D}})$ is small enough, in contrast to the non-existence results for the Dirichlet boundary data case and
$\Omega$ a star-shaped domain, see Pohozaev \cite{Po}, in the classical setting or \cite{BrCdPS} for the fractional case under the same
geometrical hypotheses. Nevertheless, and in spite of Theorem \ref{th_existencia}-$(3)$,
proceeding in a similar way as in \cite{LPT,G} we are going to show a Pohozaev-type identity for our fractional mixed Dirichlet-Neumann problems that provides us a
non-existence result under appropriate assumptions on the geometry of $\Omega,\:\Sigma_\mathcal{D},\:\Sigma_\mathcal{N}$.

Let us consider the problem
\begin{equation}
\left\{
\begin{tabular}{lcl}
$(-\Delta)^su=f(u)$ & &in $\Omega$, \\
$u>0$& & in $\Omega$,\\
$B(u)=0$  & &on $\partial\Omega=\Sigma_{\mathcal{D}}\cup\Sigma_{\mathcal{N}}$. \\
\end{tabular}
\right.
\tag{$P_f$}
\end{equation}
We have the following result.
\begin{theorem}\label{th:7}
Suppose that $u$ is a solution of problem $(P_f)$, $w=E_{s}[u]$ and $f$ is a continuous function with primitive $F$.
Then the following Pohozaev-type identity holds,
\begin{equation}\label{eq:poho}
\begin{array}{ll}
& \displaystyle(N-2s)\int_{\Omega}uf(u)dx-2N\int_{\Omega}F(u)dx\\
&\\ & =\displaystyle
\kappa_s\int_{\Sigma_{\mathcal{N}}^*}y^{1-2s}|\nabla w|^2\langle
x,\nu\rangle d\sigma(x,y)
-\kappa_s\int_{\Sigma_{\mathcal{D}}^*}y^{1-2s}|\nabla
w|^2\langle x,\nu\rangle d\sigma(x,y)\\
&\\ &\displaystyle\quad-2\int_{\Sigma_{N}}F(u)\langle x,\nu\rangle
d\sigma(x),
\end{array}
\end{equation}
where $\nu$ denotes the outwards normal vector to $\partial\Omega$.
\end{theorem}

\begin{proof}
Since $w=E_{s}[u]$ is a solution of problem
\begin{equation}
        \left\{
        \begin{tabular}{lcl}
        $-div(y^{1-2s}\nabla w)=0$ & &in $\mathcal{C}_{\Omega}$, \\
        $B^*(w)=0$  & &on $\partial_L\mathcal{C}_{\Omega}$, \\
                $\frac{\partial w}{\partial \nu^s}=f(u)$& & in $\Omega$,
        \end{tabular}
        \right.
                \tag{$P_f^*$}
\end{equation}
multiplying the equation of $(P_f^*)$ by $\varphi(x,y)$ and integrating by parts we get
\begin{equation}
\kappa_s\int_{\mathcal{C}_{\Omega}}y^{1-2s}\nabla w\nabla\varphi
dxdy=\int_{\Omega}\varphi(x,0)f(u)dx+\kappa_s\int_{\Sigma_{\mathcal{D}}^*}\varphi
y^{1-2s}\langle\nabla w,\nu^*\rangle d\sigma(x,y).
\end{equation}
With $\nu^*$ the outwards normal vector to $\partial_L\mathcal{C}_{\Omega}$. We take $\varphi(x,y)=\langle (x,y),\nabla w \rangle$ and note that $\langle\nabla w,\nu^* \rangle=|\nabla w|$ on $\Sigma_{\mathcal{D}}^*$, as well that, by construction, the outwards normal vector $\nu^*$ to the lateral boundary $\partial_L\mathcal{C}_{\Omega}$ verifies $\nu^*=(\nu,0)$ with $\nu$ the outwards normal vector to $\partial\Omega$. Then, we find,
\begin{equation*}
\begin{split}
\frac{2s-N}{2}\kappa_s\int_{\mathcal{C}_{\Omega}}y^{1-2s}|\nabla w|^2dxdy+&\frac{1}{2}\kappa_s\int_{\partial_L\mathcal{C}_{\Omega}}y^{1-2s}
|\nabla w|^2\langle x,\nu\rangle d\sigma (x,y)=\\
\int_{\Sigma_{\mathcal{N}}}F(u)\langle x,\nu\rangle d\sigma(x)-&
N\int_{\Omega}F(u)dx+\kappa_s\int_{\Sigma_{\mathcal{D}}^*} y^{1-2s}
|\nabla w|^2\langle x,\nu\rangle d\sigma(x,y),
\end{split}
\end{equation*}
which proves \eqref{eq:poho}.
\end{proof}
As a consequence we obtain a non-existence result for problem $(P_f)$.
\begin{corollary}\label{th_no}
 Assume the hypotheses of Theorem \ref{th:7} and suppose there exists $x_0\in\Omega$ such that $\langle x-x_0,\nu\rangle=0$ on $\Sigma_{\mathcal{N}}$ and $\langle x-x_0,\nu\rangle>0$ on $\Sigma_{\mathcal{D}}$. If $f$ and $F$ satisfy the inequality
$(N-2s)tf(t)-2NF(t)\geq0$, then problem $(P_f)$ has no solution.
\end{corollary}

This result highlights the difference between a mixed boundary condition problem
and a Dirichlet one as well as the relevance of the geometry of
$\Omega$ and the decomposition of $\partial\Omega$ into $\Sigma_{\mathcal{D}}$ and
$\Sigma_{\mathcal{N}}$ in the existence issues.\newline
As an example, let us consider the critical power
problem \eqref{prcrt} with $\Omega$ defined as follows. Given $A_{\alpha}$ a smooth submanifold of the unit sphere $\mathbb{S}^{N-1}$ such that $\mathcal{H}_{N-1}(A_{\alpha})=\alpha$, we set $\Omega=\{tx:x\in A_{\alpha},0<t<R\}$, $\Sigma_{\mathcal{D}}=\{x\in\overline{\Omega}: |x|=R\}$ and $\Sigma_{\mathcal{N}}=\partial{{\Omega}}\backslash\Sigma_{\mathcal{D}}$.

We consider a smooth perturbation $\widetilde{\Omega}$ where the
vertex $x_0=\overline{0}$ and the corners of $\Omega$ are regularized, such that
$|\widetilde{\Omega}\backslash\Omega|$ is small enough. Set
$\widetilde{\Sigma}_{\mathcal{D}}=\Sigma_{\mathcal{D}}$ and
$\widetilde{\Sigma}_{\mathcal{N}}=\partial{\widetilde{\Omega}}\backslash\widetilde{\Sigma}_{\mathcal{D}}$.
Then, $\langle x,\nu\rangle=0$ on
$\widetilde{\Sigma}_{\mathcal{N}}\backslash T_{\rho}$ and $\langle
x,\nu\rangle\neq0$ on
$\widetilde{\Sigma}_{\mathcal{N},\rho}=\widetilde{\Sigma}_{\mathcal{N}}\cap
T_{\rho}$ with $T_{\rho}=B_{\rho}(0)\cup\{x\in\mathbb{R}^N:R-\rho<|x|<R\}$ and some $\rho>0$ small enough, as well as $\langle x,\nu\rangle>0$ on $\widetilde{\Sigma}_{\mathcal{D}}$. Since we can approximate the cone
$\Omega$ arbitrarily by means of $\widetilde{\Omega}$ , we can let
$\rho$ be sufficiently small in order to obtain a contradiction with
the Pohozaev identity, namely
\begin{equation}\label{ewq}
\begin{array}{ll}
&\displaystyle
\frac{N-2s}{N}\int_{\widetilde{\Sigma}_{\mathcal{N},\rho}}|u|^{2^*_s}\langle
x,\nu\rangle
d\sigma\\
&\\
&\displaystyle
=\kappa_s\int_{\widetilde{\Sigma}_{\mathcal{N},\rho}^*}y^{1-2s}|\nabla
w|^2\langle x,\nu\rangle
d\sigma+R\kappa_s\int_{\widetilde{\Sigma}_{\mathcal{D}}^*}y^{1-2s}|\nabla
w|^2d\sigma.
\end{array}
\end{equation}
Thus, no solution to the problem \eqref{prcrt} exists on
$\widetilde{\Omega}$.

\begin{remark}
If we move the boundary conditions in the example above, letting
$\mathcal{H}_{N-1}(\Sigma_{\mathcal{D}})\to 0$, by means of
Theorem \ref{th_existencia}-$(3)$ we get the existence
of solution to problem \eqref{prcrt} on the perturbed cone
$\widetilde{\Omega}$. This is not in contradiction with the previous
arguments, because by this procedure, points that
belonged to the Dirichlet boundary part for which we had $\langle
x,\nu\rangle>0,$ start to contribute to the integral involving the
Neumann part of the boundary in \eqref{ewq}, and hence Theorem \ref{th_existencia}-$(3)$ and Corollary \ref{th_no} are agree.
\end{remark}

\

\noindent {\bf Acknowledgments}. The authors are partially supported
by the Ministry of Economy and Competitiveness of Spain and FEDER
under Research Project MTM2016-80618-P.

\end{document}